\documentclass{amsart}
\usepackage{amssymb}
\usepackage[all]{xy}
\usepackage{tikz}
\usepackage{multirow}
\usetikzlibrary{matrix,arrows,shapes}

\newcommand{\Z}{\mathbb{Z}}
\newcommand{\N}{\mathbb{N}}
\newcommand{\Q}{\mathbb{Q}}
\newcommand{\C}{\mathbb{C}}
\newcommand{\cK}{\mathcal{K}}
\newcommand{\cO}{\mathcal{O}}
\newcommand{\cOm}{\mathcal{O}^-}

\newcommand{\Sat}{\mathrm{Satake}}
\newcommand{\Spr}{\mathrm{Springer}}
\newcommand{\bSpr}{{\mathbf{Spr}}}
\newcommand{\Rep}{\mathrm{Rep}}
\newcommand{\Perv}{\mathrm{Perv}}
\DeclareMathOperator{\Spec}{Spec}
\newcommand{\cH}{\mathcal{H}}
\newcommand{\IC}{\mathrm{IC}}

\newcommand{\cF}{\mathcal{F}}

\newcommand{\uC}{\underline{\mathbb{C}}}

\newcommand{\GO}{G(\mathcal{O})}
\newcommand{\GOm}{G(\mathcal{O}^-)}
\newcommand{\fG}{\mathfrak{G}}
\newcommand{\HO}{H(\cO)}

\newcommand{\fg}{\mathfrak{g}}
\newcommand{\fh}{\mathfrak{h}}

\newcommand{\cM}{\mathcal{M}}
\newcommand{\cN}{\mathcal{N}}
\newcommand{\cNsm}{\mathcal{N}_{\mathrm{sm}}}
\newcommand{\cNv}{\check{\mathcal{N}}}
\newcommand{\Gr}{\mathsf{Gr}}
\newcommand{\Grsm}{\mathsf{Gr}_{\mathrm{sm}}}
\newcommand{\Grom}{\mathsf{Gr}_{0}^{-}}
\newcommand{\bo}{\mathbf{o}}
\newcommand{\cw}{\mathbf{t}}
\newcommand{\sm}{\mathrm{sm}}
\newcommand{\mn}{\mathrm{min}}
\newcommand{\pip}{\pi^\dag}

\newcommand{\Gv}{{\check G}}
\newcommand{\fgv}{\check{\mathfrak{g}}}
\newcommand{\Tv}{{\check T}}
\newcommand{\ftv}{\check{\mathfrak{t}}}
\newcommand{\lambdav}{{\check\lambda}}
\newcommand{\Lambdav}{{\check\Lambda}}
\newcommand{\Lambdavsm}{{\check\Lambda}_{\mathrm{sm}}}
\newcommand{\muv}{{\check\mu}}
\newcommand{\nuv}{{\check\nu}}
\newcommand{\alphav}{{\check\alpha}}
\newcommand{\omegav}{{\check\omega}}

\newcommand{\cP}{\mathcal{P}}
\newcommand{\Mat}{\mathrm{Mat}}

\newcommand{\simto}{\overset{\sim}{\to}}
\newcommand{\id}{\mathrm{id}}
\newcommand{\Coinv}{\mathrm{Coinv}}
\DeclareMathOperator{\rk}{rk}
\newcommand{\Hom}{\mathrm{Hom}}
\newcommand{\la}{\langle}
\newcommand{\ra}{\rangle}

\numberwithin{equation}{section}
\newtheorem{thm}{Theorem}[section]
\newtheorem{lem}[thm]{Lemma}
\newtheorem{prop}[thm]{Proposition}
\newtheorem{cor}[thm]{Corollary}
\newtheorem{conj}[thm]{Conjecture}
\theoremstyle{definition}
\newtheorem{defn}[thm]{Definition}
\theoremstyle{remark}
\newtheorem{rmk}[thm]{Remark}

\title[Satake, Springer, small]{Geometric Satake, Springer correspondence,\\ and small representations}
\author{Pramod N. Achar}
\address{Department of Mathematics\\
  Louisiana State University\\
  Baton Rouge, LA 70803\\
  U.S.A.}
\email{pramod@math.lsu.edu}
\author{Anthony Henderson}
\address{School of Mathematics and Statistics\\
  University of Sydney, NSW 2006\\
  Australia}
\email{anthony.henderson@sydney.edu.au}


\begin{document}

\begin{abstract}
For a simply-connected simple algebraic group $G$ over $\C$, we exhibit a subvariety of its affine Grassmannian that is closely related to the nilpotent cone of $G$, generalizing a well-known fact about $GL_n$.  Using this variety, we construct a sheaf-theoretic functor that, when combined with the geometric Satake equivalence and the Springer correspondence, leads to a geometric explanation for a number of known facts (mostly due to Broer and Reeder) about small representations of the dual group.
\end{abstract}

\maketitle

\section{Introduction}
\label{sect:intro}

Let $G$ be a simply-connected simple algebraic group over $\C$, and $\Gv$ its Langlands dual group. Let $T$ and $\Tv$ be corresponding maximal tori of $G$ and of $\Gv$, and let $W$ be the Weyl group of either (they are canonically identified). Recall that an irreducible representation $V$ of $\Gv$ is said to be \emph{small} if no weight of $V$ is twice a root of $\Gv$. For such $V$, the representation of $W$ on the zero weight space $V^{\Tv}$ has various special properties, mostly due to Broer and Reeder \cite{broer,reeder1,reeder2,reeder3}. 

The aim of this paper is to give a geometric explanation of these properties, using the geometric Satake equivalence (see \cite{lusztig:scf,ginzburg,mirkovicvilonen}) and the Springer correspondence (see \cite{carter}). The idea of explaining \cite{broer} using geometric Satake was suggested to us by Ginzburg; the idea of explaining \cite{reeder1} using perverse sheaves on the affine Grassmannian was suggested to Reeder by Lusztig, as mentioned in \cite{reeder2}. 

Let $\Gr$ and $\cN$ denote the affine Grassmannian and the nilpotent cone of $G$, respectively, and consider the diagram
\begin{equation} \label{eqn:sss}
\vcenter{\xymatrix@C=40pt{
\Rep(\Gv) \ar[r]^-{\Sat}_-\sim \ar[d]_{\Phi}
 & \Perv_{\GO}(\Gr) \\
\Rep(W) \ar[r]_{\Spr} & \Perv_G(\cN)}}
\end{equation}
Here $\Phi$ is the functor $V \mapsto V^\Tv \otimes \epsilon$ where $\epsilon$ denotes the sign representation of $W$.  We will construct a functor which completes diagram~\eqref{eqn:sss} to a commuting square, after restricting the top line to the subcategories corresponding to small representations.

Let $\Grsm \subset \Gr$ be the closed subvariety corresponding to small representations under geometric Satake, and let $\cM \subset \Grsm$ be the intersection of $\Grsm$ with the `opposite Bruhat cell' $\Grom$.  (See Section~\ref{sect:not} for detailed definitions.)  $\cM$ is a $G$-stable dense open subset of $\Grsm$.  Our first result is the following.

\begin{thm} \label{thm:mn}
There is an action of $\Z/2\Z$ on $\cM$, commuting with the $G$-action, and a finite $G$-equivariant map $\pi: \cM \to \cN$ that induces a bijection between $\cM/(\Z/2\Z)$ and a certain closed subvariety $\cNsm$ of $\cN$. 
\end{thm} 

The bijection mentioned in Theorem~\ref{thm:mn} is an isomorphism at least in types other than $E$; see Proposition~\ref{prop:isom}.  In type $A$, Theorem \ref{thm:mn} is well known, but usually phrased differently. In this case, $\cN$ can be embedded in $\Gr$ in two ways. $\cM$ is the union of the two embeddings and the $\Z/2\Z$-action interchanges them, so $\cNsm$ is the whole of $\cN$; see Section~\ref{subsect:a}. In general, the $\Z/2\Z$-action is related to the operation of passing from a representation of $\Gv$ to its dual, in a way which will be made precise in Remark~\ref{rmk:iota-dual}.

In type $E$, Theorem \ref{thm:mn} supplies smooth varieties mapping to certain special pieces in $\cN$, confirming a conjecture of Lusztig in at least one new case; see Proposition~\ref{prop:lusztig}.  As another application, we will establish a new characterization of small representations (the notation is defined in Section~\ref{sect:not}):

\begin{thm} \label{thm:small}
Let $V$ be an irreducible $\Gv$-representation with highest weight $\lambdav$.  Then $V$ is small if and only if $G$ acts with finitely many orbits in $\Gr_\lambdav\cap\Grom$.
\end{thm}

We can use Theorem \ref{thm:mn} to define the desired functor $\Psi:\Perv_{\GO}(\Grsm)\to\Perv_G(\cN)$. Namely, let $\Psi=\pi_*j^*$ where $\pi:\cM\to\cN$ is the map from Theorem \ref{thm:mn} and $j:\cM\hookrightarrow\Grsm$ is the inclusion. See Remark~\ref{rmk:slice} for some motivation.

\begin{thm} \label{thm:sss}
If $G$ is not of type $G_2$, then
\begin{equation*}
\vcenter{\xymatrix@C=40pt{
\Rep(\Gv)_{\sm} \ar[r]^-{\Sat}_-\sim \ar[d]_{\Phi}
 & \Perv_{\GO}(\Grsm) \ar[d]^{\Psi} \\
\Rep(W) \ar[r]_{\Spr} & \Perv_G(\cN)}}
\end{equation*}
is a commuting diagram of functors. Thus if $V \in \Rep(\Gv)$ is a small irreducible representation, we have an isomorphism of perverse sheaves
\begin{equation}\label{eqn:sssobj}
\pi_*(\Sat(V)|_{\cM}) \cong \Spr(V^{\Tv} \otimes \epsilon).
\end{equation}
\end{thm}

A uniform statement, including type $G_2$, can be obtained by slightly modifying the definition of $\Psi$; see Remark~\ref{rmk:sss-g2}. Our proof of Theorem~\ref{thm:sss} is largely empirical: the right-hand side of \eqref{eqn:sssobj} was computed by Reeder in essentially every case, and we show that the left-hand side gives the same result. 

We begin in Section~\ref{sect:not} by fixing notation, defining the map $\pi$, and proving a number of lemmas.  In Section~\ref{sect:reduce}, we state a result (Proposition~\ref{prop:general}) describing the possible behaviour of $\pi$ with respect to $G$-orbits in $\cM$ and $\cN$, and we explain how to deduce Theorems~\ref{thm:mn}--\ref{thm:sss} from this result.  Proposition~\ref{prop:general} is proved by case-by-case considerations in the classical types in Section~\ref{sect:classical}, and in the exceptional types in Section~\ref{sect:exceptional}.  Finally, Section~\ref{sect:conseq} describes some consequences of these results.  In addition to the aforementioned conjecture of Lusztig on special pieces, we explain the connection to Reeder's results on small representations~\cite{reeder1, reeder2, reeder3}, and we describe a new geometric approach to Broer's restriction theorem for covariants of small representations~\cite{broer}.

\subsubsection*{Added in revision}

After this paper appeared in preprint form, a more conceptual proof of Theorem~\ref{thm:sss} (or rather of a slightly diffferent statement, equivalent to that in Remark~\ref{rmk:sss-g2}) was found by the authors together with S.~Riche.  This proof, which is given in the sequel paper~\cite{ahr}, avoids case-by-case considerations, and applies to representations and perverse sheaves with coefficients in any Noetherian ring of finite global dimension. Note that~\cite{ahr} still uses the complex geometric set-up of the present paper, and relies on Theorem~\ref{thm:mn} above.

\subsection*{Acknowledgments}

This paper developed from discussions with V.~Ginzburg and S.~Riche, to whom the authors are much indebted.  In particular, V.~Ginzburg posed the problem of finding a geometric interpretation of Broer's covariant theorem in the context of geometric Satake.  Much of the work was carried out during a visit by P.A. to the University of Sydney in May--June 2011, supported by ARC Grant No.~DP0985184.  P.A. also received support from NSF Grant No.~DMS-1001594.

\section{Notation and preliminaries}
\label{sect:not}

The notation and conventions of Section~\ref{sect:intro} for the groups $G$, $\Gv$, $T$, $\Tv$, and $W$ remain in force throughout the paper.  For any (ind-)variety $X$ over $\C$, acted on by a (pro-)algebraic group $H$, we write $\Perv_H(X)$ for the abelian category of $H$-equivariant perverse $\C$-sheaves on $X$.  The simple perverse sheaf associated to an irreducible $H$-equivariant local system $E$ on an $H$-orbit $C \subset X$ is denoted $\IC(\overline{C},E)$, or $\IC(\overline{C})$ if $E$ is trivial.

Let $\Lambdav=\Hom(\C^\times,T)$ denote the coweight lattice of $G$, which we identify with the weight lattice of $\Gv$. Since $G$ is simply-connected, $\Lambdav$ equals the coroot lattice of $G$, i.e., the root lattice of $\Gv$. Fix a positive system, and let $\Lambdav^+ \subset \Lambdav$ be the set of dominant coweights. We have the usual partial order on $\Lambdav^+$, where $\lambdav\geq\muv$ if and only if $\lambdav-\muv$ is an $\N$-linear combination of positive roots of $\Gv$. For $\lambdav \in \Lambdav^+$, let $V_\lambdav$ denote the irreducible representation of $\Gv$ of highest weight $\lambdav$. Let $w_0$ denote the longest element of $W$, and recall that the dual representation of $V_\lambdav$ is $V_{-w_0\lambdav}$.  

Let $\Lambdavsm^+ \subset \Lambdav^+$ denote the set of \emph{small} coweights, i.e., those $\lambdav$ such that $V_\lambdav$ is a small representation of $\Gv$. Equivalently, $\lambdav\in\Lambdavsm^+$ if and only if $\lambdav\not\geq 2\alphav_0$, where $\alphav_0$ denotes the highest short root of $\Gv$. Thus, $\Lambdavsm^+$ is a lower order ideal of $\Lambdav^+$.

Next, let $\cK = \C((t))$ and $\cO = \C[[t]]$.  Any coweight $\lambdav \in \Lambdav$ gives rise to a point of $T(\cK)$ (or $G(\cK)$), denoted $\cw_\lambdav$.  Recall that the \emph{affine Grassmannian} of $G$ is the ind-variety $\Gr = G(\cK)/\GO$.  Let $\bo$ denote the base point of $\Gr$, i.e., the image of the identity element of $G(\cK)$ in $\Gr$.  For $\lambdav \in \Lambdav^+$, let $\Gr_\lambdav$ be the image in $\Gr$ of the double coset $\GO\cw_\lambdav\GO$.  This is a $\GO$-orbit in $\Gr$.  It is well known that the $\Gr_\lambdav$ are all distinct, and that every $\GO$-orbit in $\Gr$ arises in this way. The partial order on $\Lambdav^+$ corresponds to the closure order on the $\GO$-orbits, and $\Sat(V_\lambdav)=\IC(\overline{\Gr_\lambdav})$. Let $\Grsm$ be the union  of the orbits $\Gr_\lambdav$ for $\lambdav \in \Lambdavsm^+$, a closed subvariety of $\Gr$.

Let $\cOm = \C[t^{-1}] \subset \cK$, and consider the group $\GOm$. Denote the $\GOm$-orbit of $\bo$ by $\Grom$. (This is the `opposite Bruhat cell' referred to in Section~\ref{sect:intro}.)  It is known that $\Gr_\lambdav\cap\Grom$ is open dense in $\Gr_\lambdav$ for all $\lambdav$. Let
\[
\cM=\Grsm\cap\Grom.
\]
This is a $G$-stable affine open dense subvariety of $\Grsm$.  We also put
\[
\cM_\lambdav = \Gr_\lambdav \cap \Grom = \Gr_\lambdav \cap \cM
\qquad\text{for $\lambdav \in \Lambdavsm^+$.}
\]

\begin{rmk} \label{rmk:slice}
As part of Lusztig's theory of $q$-analogues of weight multiplicities~\cite{lusztig:scf}, one knows that for any $\lambdav\in\Lambdav^+$, the cohomology of the stalk of $\Sat(V_\lambdav)$ at $\bo$ is a graded vector space of total dimension equal to $\dim V_\lambdav^{\Tv}$. Since $\Gr_\lambdav\cap\Grom$ is a conical affine variety with vertex at $\bo$, the cohomology of the stalk of $\Sat(V_\lambdav)$ at $\bo$ can be identified with the hypercohomology of $\Sat(V_\lambdav)|_{\Gr_\lambdav\cap\Grom}$, as observed in~\cite[\S1.2]{bravfink}. So one could say loosely that the geometric analogue of taking the zero weight space is restricting to $\Grom$. To give a geometric interpretation of the $W$-action on the zero weight space, it is natural to want to push forward perverse sheaves on $\Grom$ to perverse sheaves on the nilpotent cone $\cN$, the home of Springer theory. As we will show, this idea works, but only for the small part $\cM$ of $\Grom$.  
\end{rmk}

Let $\fG \subset \GOm$ be the kernel of the natural map $\GOm \to G$ given by $t^{-1} \mapsto 0$.  It is easily seen that $\GOm \cong G \ltimes \fG$.  Since the stabilizer in $\GOm$ of $\bo \in \Gr$ is $G$, the action of $\GOm$ on $\Grom$ induces a $G$-equivariant isomorphism of ind-varieties
\begin{equation}\label{eqn:grom-fg-isom}
\Grom \cong \fG.
\end{equation}
Note that the natural map $\GOm \to G$ factors through $G(\C[t^{-1}]/(t^{-2}))$.  Since there is a natural identification of the Lie algebra $\fg$ of $G$ with the kernel of the map $G(\C[t^{-1}]/(t^{-2})) \to G$, we obtain a canonical homomorphism
\begin{equation}\label{eqn:fg-fg-homom}
\fG \to \fg.
\end{equation}
The $G$-equivariant morphism obtained by composing~\eqref{eqn:grom-fg-isom} and~\eqref{eqn:fg-fg-homom} is denoted
\[
\pip: \Grom \to \fg,
\]
and its restriction to $\cM$ is denoted
\[
\pi = \pip|_{\cM}: \cM \to \fg.
\]

Next, let $\theta: G(\cK) \to G(\cK)$ be the automorphism induced by the automorphism $t \mapsto -t$ of the coefficient field $\cK$.  Let $\iota: G(\cK) \to G(\cK)$ be the involutive antiautomorphism given by $\iota(g) = \theta(g^{-1})$.  The group $\fG$ is preserved by $\iota$.  Via~\eqref{eqn:grom-fg-isom}, this map induces an involution of $\Grom$, which is also denoted
\begin{equation}\label{eqn:iota-grom-defn}
\iota: \Grom \to \Grom.
\end{equation}
This map does not, in general, extend to an involution of $\Gr$.  (The map $\theta$ does induce an involution of $\Gr$, but $g \mapsto g^{-1}$ does not induce a map on $\Gr$.)  The following lemma says that $\iota$ respects the stratification of $\Grom$ induced by $\GO$-orbits on $\Gr$.  

\begin{lem} \label{lem:iota-dual}
If $x \in \Gr_\lambdav \cap \Grom$, then $\iota(x) \in \Gr_{-w_0\lambdav} \cap \Grom$.
\end{lem}
\begin{proof}
Identifying $\Grom$ with $\fG$, the assumption means that $x \in \fG$ can be written as $g\cw_\lambdav h$, with $g,h \in \GO$.  It follows that $\iota(x) = \iota(h) \iota(\cw_\lambdav) \iota(g)$.  We have $\iota(\cw_\lambdav) = \theta(\cw_{-\lambdav})$.  Since $\theta$ preserves double cosets of $\GO$, we have $\iota(x) \in \GO \cw_{-\lambdav} \GO$.  The result then follows from the observation that $-w_0\lambdav$ is the unique dominant coweight in the $W$-orbit of $-\lambdav$.
\end{proof}

In view of this lemma, we sometimes speak of a $\iota$-stable $\GO$-orbit in $\Gr$, or of two $\GO$-orbits being exchanged by $\iota$, even though $\iota$ does not extend to $\Gr$.  Note that the involution $\lambdav\mapsto -w_0\lambdav$ preserves $\Lambdavsm^+$.  Thus, $\iota$ preserves the set of $\GO$-orbits in $\Grsm$, and it induces an involution
\[
\iota: \cM \to \cM
\]
as well.  The action of $\Z/2\Z$ referred to in Theorem~\ref{thm:mn} is the one in which the nontrivial element acts by $\iota$.

\begin{rmk}\label{rmk:iota-dual}
It follows from Lemma~\ref{lem:iota-dual} that $\iota^*\Sat(V)|_{\Grom} \cong \Sat(V^*)|_{\Grom}$.
\end{rmk}

\begin{lem}\label{lem:iota-pi}
We have $\pip \circ \iota = \pip: \Grom \to \fg$.
\end{lem}
\begin{proof}
The maps $\theta$ and $g \mapsto g^{-1}$ both preserve the kernel of the map~\eqref{eqn:fg-fg-homom}, so they induce maps on $\fg$, and their composition $\bar\iota: \fg \to \fg$ has the property that $\bar\iota \circ \pip = \pip \circ \iota$.  But the inverse map on the (additive) algebraic group $\fg$ coincides with the negation map, so $\bar\iota = \id$.
\end{proof}

The varieties and maps defined above make sense for arbitrary reductive groups, not just simply-connected simple ones. Recall that for a non-simply-connected group, the affine Grassmannian $\Gr$ is not connected. In this case, one should impose the extra condition in the definition of $\Lambdavsm^+$ that $\lambdav$ belong to the coroot lattice; then $\Grsm$ lies in the connected component of $\Gr$ containing the base point $\bo$, as does $\Grom$. Indeed, $\Grsm$, $\Grom$, $\cM$, and $\fG$ can be identified with the analogous objects defined for the simply-connected cover of the derived group.

We will sometimes need to compare these constructions for different groups, and when ambiguity is possible, we include the name of the group as a subscript.  This is illustrated in the following obvious lemma, whose proof we omit.

\begin{lem}\label{lem:commute}
Let $H$ and $H'$ be two reductive groups, and let $\phi: H \to H'$ be a group homomorphism.  Then we have a commutative diagram
\[
\vcenter{\xymatrix{
\Gr_{H,0}^- \ar[r]\ar[d]_{\pip_H} & \Gr_{H',0}^- \ar[d]^{\pip_{H'}} \\
\fh \ar[r] & \fh'
}}
\]
where the horizontal maps are induced by $\phi$.
\end{lem}

Note that the inclusion of a reductive subgroup $H\hookrightarrow H'$ induces a closed embedding $\Gr_H\hookrightarrow\Gr_{H'}$.

We conclude this section with a useful algebro-geometric lemma.

\begin{lem}\label{lem:finite-map}
Let $f: Y \to X$ be a dominant morphism of affine varieties.  Let $U \subset X$ be a dense open subset whose complement has codimension at least $2$, and such that $f^{-1}(U)$ is dense in $Y$ and $f|_{f^{-1}(U)}: f^{-1}(U) \to U$ is finite.  Then $f$ is finite.
\end{lem}
\begin{proof}
Let $\hat X = \Spec \C[U]$, and let $\hat Y = \Spec \C[f^{-1}(U)]$.  (Here, $\C[Z]$ denotes the ring of regular functions on $Z$.)  We have a commutative diagram
\[
\xymatrix{
\hat Y \ar[r]^{\hat f} \ar[d]_{\mu} \ar[dr]^{q} &
   \hat X \ar[d]^{\nu} \\
Y \ar[r]_{f} & X}
\]
The assumption that $f^{-1}(U)$ is dense in $Y$ means that $\mu$ is dominant.  The map $\hat f$ is induced by $f|_{f^{-1}(U)}: f^{-1}(U) \to U$, so it is finite.  Because the complement of $U$ in $X$ has codimension at least~$2$, it follows from \cite[Exp.~VIII, Proposition~3.2]{sga2} that $\nu: \hat X \to X$ is finite.  Therefore, $q = \nu \circ \hat f$ is finite.  Since $q = f \circ \mu$ and $\mu$ is dominant, $f$ is finite also.
\end{proof}

\section{Reduction to orbit calculations}
\label{sect:reduce}

In the following two sections, we will make a careful study of the relationship between $\GO$-orbits in $\Grsm$ and $G$-orbits in the nilpotent cone $\cN$ arising from $\pi$.  This relationship involves the following notion.

\begin{defn}\label{defn:reeder}
A \emph{Reeder piece} is a subset of $\cN$ of the form $\pi(\cM_\lambdav)$ for some $\lambdav \in \Lambdavsm^+$.
\end{defn}

Here, the fact that $\pi(\cM_\lambdav)\subset\cN$ is part of the following proposition, which we will prove by case-by-case considerations in Sections~\ref{sect:classical} and~\ref{sect:exceptional}.

\begin{prop}\label{prop:general}
The variety $\cM$ is either irreducible or has two irreducible components that are exchanged by $\iota$.
The image $\cNsm=\pi(\cM)$ is an irreducible closed subset of $\cN$, and is the disjoint union of the Reeder pieces, with $\pi$ inducing a bijection
\begin{equation}\label{eqn:gen-bij}
\{ \text{$\GO$-orbits in $\Grsm$} \}/\langle \iota\rangle \overset{\sim}{\longleftrightarrow} \{\text{Reeder pieces}\}.
\end{equation}
For a Reeder piece $S$, let $\Gr_\lambdav$ and $\Gr_{-w_0\lambdav}$, which may coincide, be the corresponding $\GO$-orbits.  Then one of the following holds:
\begin{enumerate}
\item $S$ consists of a single nilpotent orbit $C$, and $\pi$ induces an isomorphism of $C$ with each of $\cM_\lambdav$ and $\cM_{-w_0\lambdav}$.  In this case, we have \label{it:gen-single}
\begin{equation}\label{eqn:zspring-single}
V_\lambdav^\Tv\otimes\epsilon \cong \Spr^{-1}(\IC(\overline{C})).
\end{equation}
\item $S$ consists of two nilpotent orbits $C_1$ and $C_2$, with $C_2 \subset \overline{C_1}$.  Then $\lambdav = -w_0\lambdav$, and $\pi$ induces an isomorphism of $C_2$ with the $\Z/2\Z$-fixed point subvariety $\cM_\lambdav^\iota$ in $\cM_\lambdav$.  On the other hand, the $\Z/2\Z$-action on $\pi^{-1}(C_1)$ is free, and the induced map $\pi^{-1}(C_1) \to C_1$ is a $2$-fold \'etale cover.  In this case, we have \label{it:gen-double}
\begin{equation}\label{eqn:zspring-double}
V_\lambdav^\Tv\otimes\epsilon \cong \Spr^{-1}(\IC(\overline{C_1})) \oplus \Spr^{-1}(\IC(\overline{C_1}, \sigma)),
\end{equation}
where $\sigma$ denotes the unique nontrivial local system of rank~$1$ on $C_1$.
\end{enumerate}
\end{prop}

If the pair $(C_1,\sigma)$ does not occur in the Springer correspondence for $G$, then the term $\Spr^{-1}(\IC(\overline{C_1}, \sigma))$ in the above formula should be understood to be $0$.  This situation only occurs for the subregular nilpotent orbit in type $G_2$; see Remark~\ref{rmk:g2-spr}.  The two possibilities can be summarized in the following diagram:
\[
\xymatrix@C=10pt{
\cM_\lambdav \ar@<.5ex>[rr]^-\iota \ar[dr]_\sim &&
  \cM_{-w_0\lambdav} \ar@<.5ex>[ll]^-\iota\ar[dl]^\sim \\
 & {}\save[]*{S = C}\restore }
\qquad
\xymatrix{
*+{\pi^{-1}(C_1)} \ar@{^{(}->}[r] \ar[d]_\sim 
  \ar@(ur,ul)[]_{\txt{\tiny $\Z/2\Z$-fixed}}&
\cM_\lambdav \ar[d]_{\pi} &
*+{\pi^{-1}(C_2)} \ar@{_{(}->}[l] \ar[d]^{\txt{\tiny 2-fold\\ \tiny \'etale}}
  \ar@(ur,ul)[]_{\txt{\tiny free $\Z/2\Z$-action}} \\
C_1 \ar@{^{(}->}[r]_{\txt{\tiny closed}} & S & 
  C_2 \ar@{_{(}->}[l]^{\txt{\tiny open}}}
\]
Explicit descriptions of the Reeder pieces, and of the bijection \eqref{eqn:gen-bij}, are given in Tables \ref{tab:classcalc}, \ref{tab:classpo}, \ref{tab:excpo} below. See Section~\ref{subsect:special} for the relationship between Reeder pieces and special pieces.

In the remainder of this section, we explain how to deduce the main theorems from Proposition~\ref{prop:general}.

\begin{proof}[Proof of Theorem~\ref{thm:mn}]
By Lemma~\ref{lem:iota-pi}, each fibre of $\pi$ is a union of $\Z/2\Z$-orbits.  But we know from Proposition~\ref{prop:general} that each $\pi^{-1}(x)$ contains just one or two points, and that in the latter case, $\iota$ exchanges the two points.  Thus, each fibre of $\pi$ consists of a single $\Z/2\Z$-orbit.  To see that $\pi$ is finite (not just quasi-finite), note that $\pi$ is finite over the open $G$-orbit $C\subset\cNsm$ by Proposition~\ref{prop:general}.  That proposition also tells us that $\Z/2\Z$ acts transitively on the components of $\cM$.  Since $\pi^{-1}(C)$ is $\iota$-stable, it is dense in $\cM$.  Since every nilpotent orbit has even dimension, the complement of $C$ in $\cNsm$ has codimension $\geq 2$, and then Lemma~\ref{lem:finite-map} implies that $\pi:\cM\to\cNsm$ is finite. 
\end{proof}

Before considering Theorem~\ref{thm:small}, we need the following lemma.

\begin{lem}\label{lem:minnotsmall}
Let $\alphav_0$ denote the highest short root of $\Gv$.  Then $\pip(\Gr_{2\alphav_0} \cap \Grom) \not\subset \cN$.
\end{lem}
\begin{proof}
We first prove the lemma in the special case where $G = SL_2$.  The coweights for $SL_2$ are in bijection with even integers, and under this bijection, we have $\alphav_0 = 2$.  Given an even integer $n$, let $\cw_n = [\begin{smallmatrix} t^{n/2} & \\ & t^{-n/2}\end{smallmatrix}] \in SL_2(\cK)$.  Now, consider the matrix
\[
g = \begin{bmatrix}
1 + t^{-1} &  t^{-2} \\ t^{-1} & 1 - t^{-1} + t^{-2}
\end{bmatrix}
\in \fG.
\]
Then $\pip(g \cdot \bo) = [\begin{smallmatrix} 1 & 0 \\ 1 & - 1\end{smallmatrix}] \notin \cN$.  On the other hand, we see from the calculation below that $g \in SL_2(\cO) \cw_4 SL_2(\cO)$, so $g \cdot \bo \in \Gr_{SL_2, 4}$:
\[
g = 
\begin{bmatrix}
& 1 \\ -1 & t^2 - t + 1
\end{bmatrix}
\begin{bmatrix}
t^2 & \\ & t^{-2}
\end{bmatrix}
\begin{bmatrix}
1 & \\ t^2 + t & 1
\end{bmatrix}.
\]
Now, for a general simply-connected group $G$, the cocharacter $\alphav_0: \C^\times \to T$ admits an extension to a homomorphism $\phi: SL_2 \to G$, where we identify $\C^\times$ with the subgroup $\{[\begin{smallmatrix} a & \\ & a^{-1}\end{smallmatrix}] \mid a \in \C^\times \} \subset SL_2$.  Extending scalars to $\cK$, we have $\phi(\cw_2) = \cw_{\alphav_0}$ and $\phi(\cw_4) = \cw_{2\alphav_0}$.  It follows that $\phi(SL_2(\cO)\cw_4 SL_2(\cO)) \subset \GO \cw_{2\alphav_0} \GO$, and thus that $\phi(g)\cdot \bo \in \Gr_{2\alphav_0} \cap \Grom$.  To prove the lemma, it suffices to show that $\pip(\phi(g)) \notin \cN$.  By Lemma~\ref{lem:commute}, $\pip(\phi(g)) = d\phi([\begin{smallmatrix} 1 & 0 \\ 1 & - 1\end{smallmatrix}])$, and the latter must be a nonzero semisimple element of $\fg$.
\end{proof}

\begin{proof}[Proof of Theorem~\ref{thm:small}]
We will actually prove that all four of the following conditions on $\lambdav\in\Lambdav^+$ are equivalent:
\begin{enumerate}
\item $G$ acts on $\Gr_\lambdav \cap \Grom$ with finitely many orbits.\label{it:small-finite}
\item The image of $\Gr_\lambdav \cap \Grom$ under $\pip$ is contained in $\cN$.\label{it:small-cn}
\item The image of $\Gr_\lambdav \cap \Grom$ under $\pip$ is contained in $\cNsm$.\label{it:small-cnsm}
\item $\lambdav\in\Lambdavsm^+$.\label{it:small-small}
\end{enumerate}
The fact that~\eqref{it:small-small} implies all the other conditions is contained in Proposition~\ref{prop:general}.  It is obvious that~\eqref{it:small-cnsm} implies~\eqref{it:small-cn}.

We now prove that~\eqref{it:small-finite} implies~\eqref{it:small-cn}.  The variety $\Gr_\lambdav \cap \Grom$, being irreducible, must contain a dense $G$-orbit $Y$.  Its image $\pip(Y)$ is a single $G$-orbit in $\fg$ that is dense in $\pip(\overline{\Gr_\lambdav} \cap \Grom)$.  In particular, the closure $\overline{\pip(Y)}$ must contain $\pip(\bo) = 0$.  The $G$-orbits in $\fg$ whose closure contains $0$ are precisely the nilpotent orbits, so $\pip(Y) \subset \cN$.  It then follows that $\pip(\overline{Y}) \subset \cN$ as well; in particular, $\pip(\Gr_\lambdav \cap \Grom) \subset \cN$.

Finally, we prove that~\eqref{it:small-cn} implies~\eqref{it:small-small}. If $\pip(\Gr_\lambdav \cap \Grom) \subset \cN$, then it follows that $\pip(\overline{\Gr_\lambdav} \cap \Grom) \subset \cN$ as well, and in view of Lemma~\ref{lem:minnotsmall}, we have $\Gr_{2\alphav_0} \not\subset \overline{\Gr_\lambdav}$.  Thus $\lambdav\not\geq 2\alphav_0$ as required.
\end{proof}

\begin{rmk}\label{rmk:infinite-orbits}
For general $\lambdav\in\Lambdavsm^+$, it is not true that $G$ acts with finitely many orbits on the whole of $\Gr_\lambdav$.
\end{rmk}

Finally, for Theorem~\ref{thm:sss}, we require some additional notation.  Let $\Perv_G(\cN)_\bSpr$ denote the Serre subcategory of $\Perv_G(\cN)$ generated by simple perverse sheaves appearing in the Springer correspondence.  Since $\Perv_G(\cN)$ is a semisimple abelian category, there is a projection functor
\[
\Perv_G(\cN) \to \Perv_G(\cN)_\bSpr, \qquad\text{denoted}\qquad \cF \mapsto \cF_\bSpr,
\]
that is exact and biadjoint to the inclusion $\Perv_G(\cN)_\bSpr \to \Perv_G(\cN)$.

\begin{proof}[Proof of Theorem~\ref{thm:sss}]
Since $j$ is an open inclusion and $\pi$ is finite, the functors $j^*$ and $\pi_*$ are both $t$-exact for perverse sheaves, and they take intersection cohomology complexes to intersection cohomology complexes.  Specifically:
\begin{enumerate}
\item If $\pi(\cM_\lambdav)$ consists of a single nilpotent orbit $C$, then
\[
\pi_*j^*\IC(\overline{\Gr_\lambdav}) \cong \IC(\overline{C}).
\]
\item If $\pi(\cM_\lambdav)$ consists of two orbits $C_1$ and $C_2$ with $C_2 \subset \overline C_1$, then
\[
\pi_*j^*\IC(\overline{\Gr_\lambdav}) \cong \IC(\overline{C_1}) \oplus \IC(\overline{C_1}, \sigma),
\]
where $\sigma$ is a nontrivial rank-$1$ local system on $C_1$.
\end{enumerate}
It then follows from Proposition~\ref{prop:general} that for any small representation $V$, we have
\begin{equation}\label{eqn:sss-pf-obj}
\Spr(V^\Tv \otimes \epsilon) \cong (\pi_*j^*\Sat(V))_\bSpr.
\end{equation}
Since $\Rep(\Gv)_\sm$ and $\Perv_G(\cN)_\bSpr$ are both semisimple $\C$-linear finite-length abel\-ian categories, the existence of such an isomorphism for each simple object in $\Rep(\Gv)_\sm$ implies that we actually have an isomorphism of functors
\begin{equation}\label{eqn:sss-pf}
\Spr \circ \Phi \cong (\cdot)_\bSpr \circ \Psi \circ \Sat.
\end{equation}
Now, every simple perverse sheaf in $\Perv_G(\cN)$ attached to a constant local system on a nilpotent orbit lies in $\Perv_G(\cN)_\bSpr$, so the projection to $\Perv_G(\cN)_\bSpr$ is necessary only if for some $\GO$-orbit $\Gr_\lambdav$ falling into case~(2) above, we have $\IC(\overline{C_1}, \sigma) \notin \Perv_G(\cN)_\bSpr$.  As noted in the remarks following Proposition~\ref{prop:general}, this happens only in type $G_2$, so in all other types, we have $\Spr \circ \Phi \cong \Psi \circ \Sat$, as desired.
\end{proof}

\begin{rmk}\label{rmk:sss-g2}
The argument above actually proves the following case-free version of Theorem~\ref{thm:sss}:  Let $\Psi': \Perv_{\GO}(\Grsm) \to \Perv_G(\cN)_\bSpr$ be the functor given by $\Psi'(\cF) = (\pi_*j^*\cF)_\bSpr$.   Then 
\begin{equation*}
\vcenter{\xymatrix@C=40pt{
\Rep(\Gv)_{\sm} \ar[r]^-{\Sat}_-\sim \ar[d]_{\Phi}
 & \Perv_{\GO}(\Grsm) \ar[d]^{\Psi'} \\
\Rep(W) \ar[r]_{\Spr} & \Perv_G(\cN)}}
\end{equation*}
is a commuting diagram of functors.
\end{rmk}

\section{The classical types}
\label{sect:classical}

In this section, we will prove Proposition \ref{prop:general} for the classical types. The result in type $A$ is essentially already known, but we spell out the argument for reference in the other types.  Table~\ref{tab:classcalc} summarizes the results of this section: for each $\lambdav \in \Lambdavsm^+$, it lists the $G$-orbits contained in $\pi(\cM_\lambdav)$. Low-rank examples are displayed in Table~\ref{tab:classpo}. As usual, coweights are written as $n$-tuples of integers $(a_1, \ldots, a_n)$, and nilpotent orbits are labelled by partitions $[b_1, \ldots, b_k]$ with $b_1 \ge \cdots \ge b_k \ge 0$.  For both weights and partitions, we use exponents to indicate multiplicities: for instance, $(2^20) = (2,2,0)$ and $[31^4] = [3,1,1,1,1]$.  As usual, we do not distinguish between partitions which differ only by adjoining zeroes at the end; thus $[31^4] = [31^40^2]$.

Let $\Mat_n$ denote the variety of $n \times n$ matrices over $\C$.  Each simply-connected classical group $G$ comes with a `standard' representation $G \to GL_n$.  Let $G'$ denote the image of this map.  (Thus, $G = G'$ if $G = SL_n$ or $Sp_n$, but $G' = SO_n$ when $G = Spin_n$.)  We can think of elements of $G'(\cK)$ as Laurent series of matrices:
\begin{equation}\label{eqn:laurent}
G'(\cK) = \left\{ \sum_{i=N}^\infty x_i t^i \,\big|\, 
\begin{array}{c}
\text{$x_i \in \Mat_n$, and the defining} \\
\text{equations for $G'$ hold}
\end{array} \right\}.
\end{equation}
In this setting, we can identify $\fG'$ (which is defined analogously to $\fG$) with the group of expressions of the form
\begin{equation}\label{eqn:fG-poly}
g = 1 + x_{-1}t^{-1} + x_{-2}t^{-2} + \cdots + x_Nt^N \in \Mat_n[t^{-1}]
\end{equation}
satisfying the definition equations for $G'$.  As mentioned in Section~\ref{sect:not}, the isogeny $G \to G'$ induces an isomorphism
\begin{equation}\label{eqn:fG-isogeny}
\fG \simto \fG',
\end{equation}
so we may think of elements of $\fG$ as expressions like~\eqref{eqn:fG-poly} as well.  For $g \in \fG$ as in~\eqref{eqn:fG-poly}, we have
\[
\pip(g \cdot \bo) = x_{-1} \in \fg.
\]

\begin{table}
\[
\begin{array}{@{}c|cl|cl|@{}}
& \lambdav \in \Lambdavsm^+ && \text{Orbits in $\pi(\cM_\lambdav)$} & \\
\cline{1-5} 
\multirow{2}{*}{$A_n$}
& (a_1, \ldots, a_n) & \text{(if $a_n \ge -1$)} & [a_1+1, a_2+1, \ldots, a_n+1] & \\
\cline{2-5}
& (a_1, \ldots, a_n) & \text{(if $a_1 \le 1$)}  & [1- a_n,\ldots,1-a_2,1-a_1] & \\
\hline
C_n
& (1^j0^{n-j}) && [2^j 1^{2n-2j}] & \\
\hline
\multirow{3}{*}{$B_n$}
&(21^{2j}0^{n-2j-1}) && [3^2 2^{2j-2} 1^{2n-4j-1}] & \text{(if $j \ge 1$)} \\
&  && [ 3 2^{2j} 1^{2n-4j-2}] & \\
\cline{2-5}
&(1^{2j}0^{n-2j}) && [2^{2j} 1^{2n-4j+1}] & \\
\hline
\multirow{5}{*}{$D_n$}
&(21^{2j}0^{n-2j-1}) & \text{(if $2j<n-1$)} & [3^2 2^{2j-2} 1^{2n-4j-2}] & \text{(if $j \ge 1$)} \\
&  && [ 3 2^{2j} 1^{2n-4j-3}] & \\
\cline{2-5}
&(21^{n-2}(\pm 1)) & \text{(if $n$ odd)} & [3^2 2^{n-3}] & \\
\cline{2-5}
& (1^{n-1}(\pm1)) & \text{(if $n$ even)} & \text{$[2^n]_I$ or $[2^n]_{II}$} & \\
\cline{2-5}
&(1^{2j}0^{n-2j}) & \text{(if $2j<n$)} & [2^{2j} 1^{2n-4j}] &
\end{array}
\]
\bigskip
\caption{$G$-orbits in $\Grsm$ and $\cNsm$ in the classical types}\label{tab:classcalc}
\end{table}

\subsection{Type $A$}
\label{subsect:a}

In this subsection, let $G=SL_n$ for some integer $n\geq 2$. We make the usual identifications
\[
\begin{split}
\Lambdav&=\{(a_1,\ldots,a_n)\in\Z^n\,|\,a_1+\cdots+a_n=0\},\\
\Lambdav^+&=\{(a_1,\ldots,a_n)\in\Lambdav\,|\,a_1\geq \cdots\geq a_n\}.
\end{split}
\]
The partial order on $\Lambdav^+$ is the usual dominance order.
Define
\[
\begin{split}
\Lambdav_{\sm,1}^+&=\{(a_1,\ldots,a_n)\in\Lambdav^+\,|\,a_n\geq -1\},\\
\Lambdav_{\sm,2}^+&=\{(a_1,\ldots,a_n)\in\Lambdav^+\,|\,a_1\leq 1\}.
\end{split}
\]
It is clear that $\Lambdav_{\sm,1}^+$ and $\Lambdav_{\sm,2}^+$ are lower order ideals of $\Lambdav^+$. 

\begin{lem}
We have $\Lambdavsm^+=\Lambdav_{\sm,1}^+\cup\Lambdav_{\sm,2}^+$.
\end{lem}
\begin{proof}
By definition, $(a_1,\ldots,a_n)\in\Lambdav^+$ is small if and only if $(a_1,\ldots,a_n)\not\geq(2,0,\ldots,0,-2)$. This condition is equivalent to saying that for some $1\leq i\leq n-1$, we have $a_1+\cdots+a_i\leq 1$. It is easy to see that, given the non-increasing condition on the $a_i$'s, this forces either $a_1\leq 1$ or $a_1+\cdots+a_{n-1}\leq 1$, the latter of which is equivalent to $a_n\geq -1$.
\end{proof}

It is easy to see that $\Lambdav_{\sm,1}^+$ is isomorphic as a poset to $\cP_n$, the poset of partitions of $n$ under the dominance order, via the map 
\begin{equation}
\tau_1:(a_1,\ldots,a_n)\mapsto[a_1+1,a_2+1,\ldots,a_n+1].
\end{equation}
Similarly, $\Lambdav_{\sm,2}^+$ is isomorphic to $\cP_n$ via the map
\begin{equation}
\tau_2:(a_1,\ldots,a_n)\mapsto[1-a_n,1-a_{n-1},\ldots,1-a_1].
\end{equation} 
In particular, $\Lambdav_{\sm,1}^+$ has a unique maximal element $(n-1,-1,\ldots,-1)$, and $\Lambdav_{\sm,2}^+$ has a unique maximal element $(1,\ldots,1,1-n)$.

If $\lambdav\in\Lambdav_{\sm,1}^+$, then $-w_0\lambdav=\tau_2^{-1}(\tau_1(\lambdav))\in\Lambdav_{\sm,2}^+$. Hence the involution $\lambdav\mapsto-w_0\lambdav$ interchanges $\Lambdav_{\sm,1}^+$ and $\Lambdav_{\sm,2}^+$, and fixes every element of their intersection. Note that $\tau_i(\Lambdav_{\sm,1}^+\cap\Lambdav_{\sm,2}^+)$ is the set of partitions in $\cP_n$ with largest part${}\leq 2$. In summary, the poset $\Lambdavsm^+$ is obtained by taking two copies of $\cP_n$ and gluing them together along the lower order ideal of partitions with largest part${}\leq 2$.

We let $\cM_i=\bigcup_{\lambdav\in\Lambdav_{\sm,i}^+}\cM_\lambdav$, for $i=1,2$. By the preceding paragraph and Lemma~\ref{lem:iota-dual}, we have the following.

\begin{lem}\label{lem:A-comp}
$\cM_1$ and $\cM_2$ are the irreducible components of $\cM$ (or, if $n=2$, $\cM_1=\cM_2=\cM$).  The involution $\iota$ interchanges $\cM_1$ and $\cM_2$.\qed
\end{lem}

Let $\lambdav=(a_1,\cdots,a_n)\in\Lambdav$.  In the setting of~\eqref{eqn:laurent}, the element $\cw_\lambdav$ can be written as $\sum_{j=1}^n e_{jj}t^{a_j}$, where $e_{jj}$ is the usual matrix unit.
\begin{lem} \label{lem:laurent}
Let $g=\sum_{i=N}^\infty x_i t^i\in G(\cK)$, where $x_{N}\neq 0$. Let $\lambdav=(a_1,\ldots,a_n)\in\Lambdav^+$ be such that $g\cdot\bo\in\Gr_\lambdav$.
\begin{enumerate}
\item $N=a_n$.
\item The rank of $x_N$ equals the number of $j$ such that $a_j=a_n$.
\item More generally, for any $s\geq 1$, the rank of the $sn\times sn$ matrix
\[
\begin{bmatrix}
x_N&x_{N+1}&\cdots&x_{N+s-2}&x_{N+s-1}\\
0&x_N&\cdots&x_{N+s-3}&x_{N+s-2}\\
\vdots&\vdots&\ddots&\vdots&\vdots\\
0&0&\cdots&x_N&x_{N+1}\\
0&0&\cdots&0&x_N
\end{bmatrix}
\]
equals $\sum_{j=1}^n \max\{s-(a_j-a_n),0\}$.
\end{enumerate}
\end{lem} 
\begin{proof}
It is easy to see that the leading power $N$ and the ranks of the matrices in the statement are constant on the double coset $\GO g \GO$. So we can assume that $g=\cw_\lambdav$, in which case the claims are easy.
\end{proof}

\begin{prop} \label{prop:A-M}
The irreducible components of $\cM$, and their intersection, are described by:
\[ 
\begin{split}
\cM_1&=\{(1+xt^{-1})\cdot\bo\,|\,x\in\cN\},\\
\cM_2&=\{(1-xt^{-1})^{-1}\cdot\bo\,|\,x\in\cN\},\\
\cM_1\cap\cM_2&=\{(1+xt^{-1})\cdot\bo\,|\,x\in\cN, x^2=0\}.
\end{split}
\]
In particular, $\cM_1\cap\cM_2$ equals the fixed-point subvariety $\cM^\iota$.
\end{prop}
\begin{proof}
Let $g\in\fG$.  Thus, $g$ is an expression of the form~\eqref{eqn:fG-poly} with $\det(g) = 1$.  Let $\lambdav=(a_1,\ldots,a_n)\in\Lambdav^+$ be such that $g\cdot\bo\in\Gr_\lambdav$. By Lemma \ref{lem:laurent}(1), $\lambdav\in\Lambdav_{\sm,1}^+$ if and only if $g=1+xt^{-1}$ for some $x\in\Mat_n$. Clearly $1+xt^{-1}$ belongs to $\fG$ if and only if $x$ belongs to the nilpotent cone $\cN$, so we obtain the stated description of $\cM_1$. The description of $\cM_2$ follows, because $\cM_2=\iota(\cM_1)$. The rest is then clear.
\end{proof}

As an immediate consequence, $\cNsm=\cN$ and $\pi$ restricted to $\cM_i$ gives an isomorphism $\cM_i\simto\cN$ for $i=1,2$. It is well known that the $G$-orbits in $\cN$ are in bijection with $\cP_n$, via Jordan form. In particular, the number of $G$-orbits in $\cN$ equals $|\Lambdav_{\sm,i}^+|$. Therefore each $\cM_\lambdav$ for $\lambdav\in\Lambdavsm^+$ is a single $G$-orbit, and each Reeder piece in $\cN$ is a single orbit. In fact, we have:

\begin{prop} \label{prop:A-orbits}
For $\lambdav\in\Lambdav_{\sm,i}^+$, $\pi(\cM_\lambdav)$ is the nilpotent orbit labelled by the partition $\tau_i(\lambdav)$.
\end{prop}
\begin{proof}
Since $\pi\circ\iota=\pi$, we can assume that $i=1$. We need to show that if $x\in\cN$ has Jordan type $[b_1,\cdots,b_n]$, then $(1+xt^{-1})\cdot\bo\in\Gr_{(b_1-1,\cdots,b_n-1)}$. This is trivial if $x=0$, so we can assume that $x\neq 0$, and therefore $b_n=0$. By Lemma \ref{lem:laurent}(3), it suffices to show that for any $s\geq 1$, the rank of the $sn\times sn$ matrix 
\[
\begin{bmatrix}
x&1&0&\cdots&0\\
0&x&1&\cdots&0\\
0&0&x&\cdots&0\\
\vdots&\vdots&\vdots&\ddots&\vdots\\
0&0&0&\cdots&x
\end{bmatrix}
\]
equals $\sum_{j=1}^n \max\{s-b_j,0\}$. But the rank of this matrix is clearly equal to 
\[ sn-\dim\ker(x^s)=sn-\sum_{j=1}^n \min\{b_j,s\}, \]
as required.
\end{proof}

\begin{rmk}
In \cite[Section 2]{lusztig:green}, Lusztig defined an embedding of the nilpotent cone of $GL(\overline{V})$ in the affine Grassmannian of $GL(V)$, where $\overline{V}$ and $V$ are $n$-dimensional vector spaces in duality with each other. This gives two embeddings of the nilpotent cone of $GL_n$ in the affine Grassmannian of $GL_n$, depending on the choice of which of $\overline{V}$ or $V$ to identify with $\C^n$. Of course, the nilpotent cone of $GL_n$ is the same as the $\cN$ of this section, and the relevant connected components of the affine Grassmannian of $GL_n$ can each be identified with our $\Gr$. The resulting two embeddings of $\cN$ in $\Gr$ are precisely the isomorphisms $\cN\simto\cM_1:x\mapsto (1+xt^{-1})\cdot\bo$ and $\cN\simto\cM_2:x\mapsto (1-xt^{-1})^{-1}\cdot\bo$.
\end{rmk}

\begin{proof}[Proof of Proposition~\ref{prop:general} in type $A$]
We have seen that the first two sentences of the statement are true, and that case (1) holds always. All that remains to prove is that for any $\lambdav\in\Lambdavsm^+$, the representation of the symmetric group $S_n$ on $V_\lambdav^\Tv$ is as claimed. It suffices to check this for $\lambdav\in\Lambdav_{\sm,1}^+$, where the statement is that $V_\lambdav^\Tv$ is the irreducible representation labelled by the partition $\tau_1(\lambdav)$ tensored with the sign representation. Now as a representation of $GL_n$, $V_\lambdav$ is the irreducible representation with highest weight $\tau_1(\lambdav)$ tensored with the one-dimensional representation $\det^{-1}$. So the claim follows from Schur--Weyl duality.
\end{proof}

\subsection{Type $C$}
\label{subsect:c}

In this subsection, let $G=Sp_{2n}$ for some integer $n\geq 2$. We make the usual identifications
\[
\Lambdav=\Z^n,\quad \Lambdav^+=\{(a_1,\ldots,a_n)\in\Z^n\,|\,a_1\geq\cdots\geq a_n\geq 0\}.
\] 
Note that under the embedding $G\subset SL_{2n}$, with suitable choices of maximal tori and positive systems, a dominant coweight $(a_1,\ldots,a_n)\in\Lambdav^+$ for $G$ maps to the dominant coweight $(a_1,\ldots,a_n,-a_n,\ldots,-a_1)$ for $SL_{2n}$.

\begin{lem} \label{lem:C-small}
We have $\Lambdavsm^+=\{(1^j0^{n-j})\,|\,0\leq j\leq n\}$.  Moreover, $\cM$ is irreducible.
\end{lem}
\begin{proof}
By definition, $(a_1,\ldots,a_n)\in\Lambdav^+$ is small if and only if $(a_1,\ldots,a_n)\not\geq(2,0,\ldots,0)$, which is clearly equivalent to $a_1\leq 1$.

Obviously the partial order on $\Lambdavsm^+$ is a total order in this case, with maximal element $(1^n)$. Hence $\cM$ is irreducible.
\end{proof}

\begin{prop} \label{prop:C-M}
We have $\cM=\{(1+xt^{-1})\cdot\bo\,|\,x\in\cN, x^2=0\}$. In particular, $\iota$ acts trivially on $\cM$.
\end{prop} 
\begin{proof}
Let $g\in\fG$ be as in~\eqref{eqn:fG-poly}, and let $\lambdav=(a_1,\cdots,a_n)\in\Lambdav^+$ be such that $g\cdot\bo\in\Gr_\lambdav$. Then as a point in the affine Grassmannian of $SL_{2n}$, $g\cdot\bo$ belongs to the orbit labelled by $\muv=(a_1,\cdots,a_n,-a_n,\cdots,-a_1)$. By Lemma \ref{lem:C-small}, $\lambdav\in\Lambdavsm^+$ if and only if $\muv$ lies in the intersection $\Lambdav_{\sm,1}^+\cap\Lambdav_{\sm,2}^+$ defined in the previous subsection (for $SL_{2n}$ rather than for $SL_n$). Using the description of $\cM_1\cap\cM_2$ in Proposition \ref{prop:A-M}, we deduce that $g\cdot\bo\in\cM$ if and only if $g$ has the form $1+xt^{-1}$ for some $x\in\cN$ such that $x^2=0$. Moreover, for any $x\in\cN$ such that $x^2=0$, $1+xt^{-1}=\exp(xt^{-1})\in\fG$. 
\end{proof}

As an immediate consequence, $\cNsm=\{x\in\cN\,|\,x^2=0\}$ and we have an isomorphism $\pi:\cM\simto\cNsm$. It is well known that the $G$-orbits in $\cNsm$ are in bijection with the partitions of $2n$ with largest part $\leq 2$, via Jordan form. In particular, the number of $G$-orbits in $\cNsm$ equals $|\Lambdavsm^+|$. Therefore each $\cM_\lambdav$ for $\lambdav\in\Lambdavsm^+$ is a single $G$-orbit, and each Reeder piece in $\cN$ is a single orbit. In fact, we have:
\begin{prop} \label{prop:C-orbits}
We have $\pi(\cM_{(1^j0^{n-j})})=[2^j1^{2n-2j}]$.
\end{prop}
\begin{proof}
This is automatic, because the closure order on $G$-orbits in either $\cM$ or $\cNsm$ is a total order. Alternatively, it follows from Lemma \ref{lem:laurent}(2).
\end{proof}

\begin{proof}[Proof of Proposition~\ref{prop:general} in type $C$]
We have seen that the first two sentences of the statement are true, and that case (1) holds always. All that remains to prove is that for any $(1^j0^{n-j})\in\Lambdavsm^+$, the representation of the Weyl group $W=W(B_n)$ on $V_{(1^j0^{n-j})}^\Tv$ is as claimed. As a representation of $SO_{2n+1}$, $V_{(1^j0^{n-j})}\cong\wedge^j(\C^{2n+1})$. It is straightforward to verify that the representation of $W$ on the zero weight space of      $\wedge^j(\C^{2n+1})$ is the irreducible labelled by the bipartition $((n-\frac{j}{2});(\frac{j}{2}))$ if $j$ is even or $((\frac{j-1}{2});(n-\frac{j-1}{2}))$ if $j$ is odd. After tensoring with sign, this becomes the irreducible labelled by $((1^{j/2});(1^{n-j/2}))$ if $j$ is even or $((1^{n-(j-1)/2});(1^{(j-1)/2}))$ if $j$ is odd, which does indeed correspond to the trivial local system on the orbit $[2^j1^{2n-2j}]$ under the Springer correspondence, as observed by Reeder in \cite[Table 5.1]{reeder3}.
\end{proof}

\begin{table}
\tikzstyle{reed1}=[draw,inner sep=0pt,shape=ellipse,minimum height=.75cm, minimum width=0.75cm]
\tikzstyle{reed2}=[draw,shape=ellipse,rotate=-40,minimum height=1cm,minimum width=3cm]
\tikzstyle{lreed2}=[draw,shape=ellipse,rotate=38,minimum height=1cm,minimum width=3.2cm]
\tikzstyle{pimap}=[dashed,->]
\begin{center}
\small
\begin{tabular}{cc}
\begin{tikzpicture}[x=.9cm]
  \path (-4,0)  node (w4a)  {$3(-1)^3$};
  \path (-2,0)  node (w4b)  {$1^3(-3)$};
  \path (-4,-1) node (w3a) {$20(-1)^2$};
  \path (-2,-1) node (w3b) {$1^20(-2)$};
  \path (-3,-2) node (w22) {$1^2(-1)^2$};
  \path (-3,-3) node (w2) {$10^2(-1)$};
  \path (-3,-4) node (w1) {$0^4$};

  \path (0,0)   node[reed1] (n4)  {$4$};
  \path (0,-1)  node[reed1] (n3) {$31$};
  \path (0,-2)  node[reed1] (n22) {$2^2$};
  \path (0,-3)  node[reed1] (n2) {$21^2$};
  \path (0,-4)  node[reed1] (n1)  {$1^4$};
  
  \draw (n4) -- (n3) -- (n22) -- (n2) -- (n1);
  \draw (w4a) -- (w3a) -- (w22) -- (w2) -- (w1)
        (w4b) -- (w3b) -- (w22);
  \draw[pimap] (w4a)  .. controls(-2,0.5) .. (n4);
  \draw[pimap] (w4b)  .. controls(-1,-0.25) .. (n4);
  \draw[pimap] (w3a) .. controls(-2,-0.5) .. (n3);
  \draw[pimap] (w3b) .. controls(-1,-1.25) .. (n3);
  \draw[pimap] (w22) -- (n22);
  \draw[pimap] (w2) -- (n2);
  \draw[pimap] (w1) -- (n1);
\end{tikzpicture}
&
\begin{tikzpicture}[x=.9cm]
  \path (-2,0)  node (w4) {$1^4$};
  \path (-2,-1) node (w3) {$1^30$};
  \path (-2,-2) node (w2) {$1^20^2$};
  \path (-2,-3) node (w1) {$1^10^3$};
  \path (-2,-4) node (w0) {$0^4$};

  \path (0,0)  node[reed1] (n4) {$2^4$};
  \path (0,-1) node[reed1] (n3) {$2^31^2$};
  \path (0,-2) node[reed1] (n2) {$2^21^4$};
  \path (0,-3) node[reed1] (n1) {$21^6$};
  \path (0,-4) node[reed1] (n0) {$1^8$};
  
  \draw (n4) -- (n3) -- (n2) -- (n1) -- (n0);
  \draw (w4) -- (w3) -- (w2) -- (w1) -- (w0);
  \draw[pimap] (w4) -- (n4);
  \draw[pimap] (w3) -- (n3);
  \draw[pimap] (w2) -- (n2);
  \draw[pimap] (w1) -- (n1);
  \draw[pimap] (w0) -- (n0);
\end{tikzpicture}
\\
$A_3$ & $C_4$ \\
\ \\
\begin{tikzpicture}[x=.9cm]
  \path (-2,0)   node (w33)  {$21^20$};
  \path (-1,-2)  node (w2222) {$1^4$};
  \path (-2,-3)  node (w3)  {$20^3$};
  \path (-1,-4)  node (w22)  {$1^20^2$};
  \path (-1,-5)  node (w1)   {$0^4$};

  \path (0,0)   node (n33)  {$3^21^3$};
  \path (1,-1)  node (n322) {$32^21^2$};
  \path (2,-2)  node[reed1] (n2222) {$2^41$};
  \path (1,-3)  node[reed1] (n3)  {$31^6$};
  \path (2,-4)  node[reed1] (n22)  {$2^21^5$};
  \path (2,-5)  node[reed1] (n1)   {$1^9$};
  
  \path (0.5,-0.5) coordinate[reed2] (e1);
  
  \draw (n33) -- (n322) -- (n2222) -- (n22) -- (n1)  (n322) -- (n3) -- (n22);
  \draw (w33) -- (w2222) -- (w22) -- (w1)  (w33) -- (w3) -- (w22);
  \draw[pimap] (w33) .. controls(-1,0) .. (e1);
  \draw[pimap] (w2222) -- (n2222);
  \draw[pimap] (w22) -- (n22);
  \draw[pimap] (w3) -- (n3);
  \draw[pimap] (w1) -- (n1);
\end{tikzpicture}
&
\begin{tikzpicture}[x=.9cm]
  \path (-2,0)   node (w33)  {$21^20$};
  \path (-1.5,-2)  node (w2222a) {$1^4$};
  \path (-0.5,-2)  node (w2222b) {$1^3(-1)$};
  \path (-2,-3)  node (w3)  {$20^3$};
  \path (-1,-4)  node (w22)  {$1^20^2$};
  \path (-1,-5)  node (w1)   {$0^4$};

  \path (0,0)   node (n33)  {$3^21^2$};
  \path (1,-1)  node (n322) {$32^21$};
  \path (1.5,-2)  node[reed1] (n2222a) {$2^4_I$};
  \path (2.5,-2)  node[reed1] (n2222b) {$2^4_{II}$};
  \path (1,-3)  node[reed1] (n3)  {$31^5$};
  \path (2,-4)  node[reed1] (n22)  {$2^21^4$};
  \path (2,-5)  node[reed1] (n1)   {$1^8$};
  
  \path (0.5,-0.5) coordinate[reed2] (e1);
  
  \draw (n33) -- (n322) -- (n2222a) -- (n22) -- (n1)
        (n322) -- (n2222b) -- (n22)
        (n322) -- (n3) -- (n22);
  \draw (w33) -- (w2222a) -- (w22) -- (w1)
        (w33) -- (w2222b) -- (w22)
        (w33) -- (w3) -- (w22);
  \draw[pimap] (w33) .. controls(-1,0) .. (e1);
  \draw[pimap] (w2222a) .. controls(-0.5,-1.4) .. (n2222a);
  \draw[pimap] (w2222b) .. controls(1.5,-2.6) .. (n2222b);
  \draw[pimap] (w22) -- (n22);
  \draw[pimap] (w3) -- (n3);
  \draw[pimap] (w1) -- (n1);
\end{tikzpicture}
\\
$B_4$ & $D_4$
\end{tabular}
\bigskip

\end{center}
\caption{$\Grsm$, $\cNsm$, and Reeder pieces in some low-rank classical groups}\label{tab:classpo}
\end{table}

\subsection{Type $B$}
\label{subsect:b}

In this subsection, let $G=Spin_{2n+1}$ for some integer $n\geq 3$. We make the usual identifications 
\[
\begin{split}
\Lambdav&=\{(a_1,\ldots,a_n)\in\Z^n\,|\,a_1+\cdots+a_n\in 2\Z\},\\
\Lambdav^+&=\{(a_1,\ldots,a_n)\in\Lambdav\,|\,a_1\geq \cdots\geq a_n\geq 0\}.
\end{split}
\]
Under the map $G\to SL_{2n+1}$, with suitable choices of maximal tori and positive systems, the dominant coweight $(a_1,\ldots,a_n)\in\Lambdav^+$ for $G$ maps to the dominant coweight $(a_1,\ldots,a_n,0,-a_n,\ldots,-a_1)$ for $SL_{2n+1}$.

\begin{lem} \label{lem:B-small}
We have 
\[\textstyle
\Lambdavsm^+=\{(21^{2j}0^{n-2j-1})\,|\,0\leq j\leq\lfloor\frac{n-1}{2}\rfloor\}\cup\{(1^{2j}0^{n-2j})\,|\,0\leq j\leq\lfloor\frac{n}{2}\rfloor\}.
\]
Moreover, $\cM$ is irreducible.
\end{lem}
\begin{proof}
By definition, $(a_1,\ldots,a_n)\in\Lambdav^+$ is small if and only if $(a_1,\ldots,a_n)\not\geq(2,2,0,\ldots,0)$, which is easily seen to be equivalent to $a_1\leq 2$ and $a_1+a_2\leq 3$.

The partial order on $\Lambdavsm^+$ is described as follows. The elements $(1^{2j}0^{n-2j})$ form a chain in the obvious way, as do the elements $(21^{2j}0^{n-2j-2})$. The only other covering relations are that for $j\geq 1$, $(1^{2j}0^{n-2j})$ is covered by $(21^{2j-2}0^{n-2j+1})$. In particular, $\Lambdavsm^+$ has unique maximal element $(21^{n-1})$ if $n$ is odd or $(21^{n-2}0)$ if $n$ is even. Hence $\cM$ is irreducible.
\end{proof}

A crucial point is that under the map $G\to SL_{2n+1}$, the small coweights of the form $(21^{2j}0^{n-2j-1})$ map to non-small coweights for $SL_{2n+1}$. So we cannot simply use Proposition \ref{prop:A-M} to describe $\cM$ in type $B$, as we did in type $C$. However, if we let $\cM'=\bigcup_{j}\cM_{(1^{2j}0^{n-2j})}$, then $\cM'$ is a closed subvariety of $\cM$ which is analogous to $\cM$ in type $C$.

An element of $\fG$ is an expression as in~\eqref{eqn:fG-poly} that satisfies the defining equations for $SO_{2n+1}$, i.e., that preserves some nondegenerate symmetric bilinear form $(\cdot,\cdot)$ on $\C^{2n+1}$.  If $g = 1 + xt^{-1} + yt^{-2} + \cdots \in \fG$, then $\pip(g\cdot\bo)=x$ must belong to the Lie algebra $\fg$, but $y$ need not. Recall that $\fg$ consists of the elements of $\Mat_{2n+1}$ which are anti-self-adjoint with respect to $(\cdot,\cdot)$.

\begin{prop} \label{prop:B-M}
Write $\cM$ as the disjoint union $\cM'\cup\cM''\cup\cM'''$, where $\cM'$ is as above and $\cM''=(\cM\setminus\cM')^\iota$.
\begin{enumerate}
\item We have $\cM'=\{(1+xt^{-1})\cdot\bo\,|\,x\in\cN, x^2=0\}$, and $\iota$ acts trivially on $\cM'$. Hence $\cM'\cup\cM''=\cM^\iota$.
\item We have
\[\textstyle
\cM''=\{(1+xt^{-1}+\frac{1}{2}x^2t^{-2})\cdot\bo\,|\,x\in\cN, x^3=0, \rk(x^2)=1\}.
\]
\item We have
\[
\cM'''=\{(1+xt^{-1}+yt^{-2})\cdot\bo\,|\,x\in\cN, x^3=0, \rk(x^2)=2, y\in\{(x^2)_1,(x^2)_{2}\}\},
\] 
where $(x^2)_1\neq(x^2)_{2}$ are uniquely defined up to order by the conditions
\[
\begin{split}
&(x^2)_1+(x^2)_{2}=x^2,\ \rk((x^2)_1)=\rk((x^2)_{2})=1,\text{ and}\\ 
&(x^2)_1\text{ and }(x^2)_{2}\text{ are adjoint to each other for }(\cdot,\cdot), 
\end{split}
\] 
and $\iota$ acts on $\cM'''$ by interchanging $(1+xt^{-1}+(x^2)_1 t^{-2})\cdot\bo$ and $(1+xt^{-1}+(x^2)_{2}t^{-2})\cdot\bo$.
\end{enumerate}
\end{prop} 
\begin{proof}
Part (1) is proved in the same way as Proposition \ref{prop:C-M}. For $g\in\fG$, we have $g\cdot\bo\in\cM\setminus\cM'$ if and only if, as a point in the affine Grassmannian of $SL_{2n+1}$, $g\cdot\bo$ belongs to the orbit labelled by a coweight $(a_1,\ldots,a_{2n+1})$ where $a_{2n+1}=-2$ and $a_{2n}>-2$. Using Lemma \ref{lem:laurent}(1)(2), we deduce that
\begin{equation} \label{eqn:rank-1}
\cM''\cup\cM'''=\{(1+xt^{-1}+yt^{-2})\cdot\bo\,|\,1+xt^{-1}+yt^{-2}\in\fG, \rk(y)=1\}.
\end{equation}
Now $(1+xt^{-1}+yt^{-2})\cdot\bo$ is fixed by $\iota$ if and only if
\[
1=(1+xt^{-1}+yt^{-2})(1-xt^{-1}+yt^{-2})=1+(2y-x^2)t^{-2}+(xy-yx)t^{-3}+y^2t^{-4},
\]
which implies $y=\frac{1}{2}x^2$ and $x^4=0$, so $x\in\cN$. Since $\rk(x^2)=1$, we must in fact have $x^3=0$. Conversely, if $x\in\cN$, $x^3=0$, and $\rk(x^2)=1$, then $1+xt^{-1}+\frac{1}{2}x^2t^{-2}=\exp(xt^{-1})\in\fG$ and $(1+xt^{-1}+\frac{1}{2}x^2t^{-2})\cdot\bo$ is fixed by $\iota$. This proves part (2).

To prove part (3), we start by explaining the definition of $(x^2)_1,(x^2)_{2}$ in the right-hand side. Let $x\in\cN$ be such that $x^3=0$ and $\rk(x^2)=2$. Let $U$ be the image of $x^2$. Since $x^2$ is self-adjoint for $(\cdot,\cdot)$, the subspace $U^\perp$ perpendicular to $U$ equals the kernel of $x^2$. We can define a bilinear form $\langle\cdot,\cdot\rangle$ on $U$ uniquely by the rule $\langle x^2 v,u\rangle=(v,u)$ for any $v\in\C^{2n+1}$, $u\in U$. It is easy to check that this form is symmetric and nondegenerate. Hence there are exactly two $1$-dimensional subspaces $L\subset U$ which are isotropic for $\langle\cdot,\cdot\rangle$. In terms of the original bilinear form $(\cdot,\cdot)$, this means that there are exactly two $1$-dimensional subspaces $L\subset U$ such that $L^\perp=(x^2)^{-1}(L)$. Call these $1$-dimensional subspaces $L_1$ and $L_{2}$ in some order.

Since $L_1^\perp+L_{2}^\perp=\C^{2n+1}$ and $L_1^\perp\cap L_{2}^\perp=U^\perp=\ker(x^2)$, we can uniquely write $x^2=(x^2)_1+(x^2)_{2}$ where $(x^2)_1$ vanishes on $L_{2}^\perp$ and $(x^2)_{2}$ vanishes on $L_{1}^\perp$. The images of $(x^2)_1$ and $(x^2)_{2}$ equal $L_1$ and $L_{2}$ respectively. If $(x^2)_1',(x^2)_{2}'$ denote the adjoints of $(x^2)_1,(x^2)_{2}$ for $(\cdot,\cdot)$, then we have $x^2=(x^2)_1'+(x^2)_{2}'$ where $(x^2)_1'$ vanishes on $L_{1}^\perp$ and $(x^2)_{2}'$ vanishes on $L_{2}^\perp$, and hence $(x^2)_1'=(x^2)_{2}$, $(x^2)_{2}'=(x^2)_{1}$. Conversely, it is easy to see that if $x^2=y+y'$ where $y$ and $y'$ have rank $1$ and are adjoint to each other for $(\cdot,\cdot)$, then the images of $y$ and $y'$ satisfy the defining property of $L_1$ and $L_{2}$, and therefore $y$ and $y'$ must equal $(x^2)_1$ and $(x^2)_{2}$ in some order.
 
Now the assumption $x^3=0$ means that $x(U)=0$, from which we deduce that $x(x^2)_{2}=x(x^2)_1=0$, and hence (by taking adjoints) $(x^2)_1x=(x^2)_{2}x=0$. Moreover, $U\subset U^\perp$, so $(x^2)_1(x^2)_{2}=(x^2)_{2}(x^2)_{1}=0$. We conclude that 
\[
\begin{split}
(1+xt^{-1}+(x^2)_{1}t^{-2})(1-xt^{-1}+(x^2)_{2}t^{-2})&=1,\\
(1-xt^{-1}+(x^2)_{2}t^{-2})(1+xt^{-1}+(x^2)_{1}t^{-2})&=1.
\end{split}
\]
Since the inverse of $1+xt^{-1}+(x^2)_{1}t^{-2}$ equals its adjoint, it belongs to $\fG$, as does $1+xt^{-1}+(x^2)_{2}t^{-2}=\iota(1+xt^{-1}+(x^2)_{1}t^{-2})$. Taking into account \eqref{eqn:rank-1}, we see that $(1+xt^{-1}+(x^2)_{1}t^{-2})\cdot\bo,(1+xt^{-1}+(x^2)_{2}t^{-2})\cdot\bo\in\cM'''$ as claimed. 

Finally, we must show that every element of $\cM'''$ is obtained in this way. By \eqref{eqn:rank-1}, any element of $\cM'''$ has the form $g\cdot\bo$ where $g=1+xt^{-1}+yt^{-2}\in\fG$ is such that $\rk(y)=1$ and $g\cdot\bo$ is not fixed by $\iota$. From Lemma \ref{lem:iota-dual}, we know that $\iota(g\cdot\bo)$ also belongs to $\cM'''$, so we must have $g^{-1}=1-xt^{-1}+y't^{-2}$ where $\rk(y')=1$ and $y'\neq y$. The equations
\[
(1+xt^{-1}+yt^{-2})(1-xt^{-1}+y't^{-2})=1=(1-xt^{-1}+y't^{-2})(1+xt^{-1}+yt^{-2})
\]
imply that 
\[ y+y'=x^2,\ yx=xy',\ xy=y'x,\ yy'=y'y=0. \]
Since $g\in\fG$, we know that $y$ and $y'$ are adjoint to each other for $(\cdot,\cdot)$. If $L$ and $L'$ denote the $1$-dimensional images of $y$ and $y'$ respectively, then $\ker(y)=(L')^\perp$ and $\ker(y')=L^\perp$. Moreover, $L\subset L^\perp$ because $y'y=0$, and similarly $L'\subset(L')^\perp$. It is not possible that $L=L'$, because that would force $y'=-y$, which would lead to the contradictory conclusions $y\in\fg$, $y^2=0$, and $\rk(y)=1$. So $U=L+L'$ is $2$-dimensional, and equals the image of $y+y'=x^2$. Since $x^2$ is self-adjoint, $\ker(x^2)=U^\perp$.

All that remains is to show that $x^3=0$. Knowing that $\rk(x^2)=2$, it suffices to show that $x^4=0$, because there are no elements of $\cN$ with a single Jordan block of size $4$. So we need only show that $U\subset U^\perp$. If $U\not\subset U^\perp$, then the restriction of $(\cdot,\cdot)$ to $U$ is nondegenerate, and $L$ and $L'$ are the two isotropic lines for that restriction. But from the equations $yx=xy'$ and $xy=y'x$ we see that $x(L)\subseteq L'$, $x(L')\subseteq L$. Since $x$ is anti-self-adjoint, this forces the restriction of $x$ to $U$ to be zero, giving the contradictory conclusion that $U\subset U^\perp$ after all.
\end{proof}

As an immediate consequence, $\cNsm=\{x\in\cN\,|\,x^3=0,\rk(x^2)\leq 2\}$. It is well known that the $G$-orbits in $\cN$ are parametrized by their Jordan types, which are the partitions of $2n+1$ in which every even part has even multiplicity. The orbits belonging to $\cNsm$ are as follows:
\begin{equation} \label{eqn:B-orbitlist}
\begin{split}
&[2^{2j}1^{2n-4j+1}],\text{ for }0\leq j\leq \textstyle\lfloor\frac{n}{2}\rfloor,\\
&[32^{2j}1^{2n-4j-2}],\text{ for }0\leq j\leq \textstyle\lfloor\frac{n-1}{2}\rfloor,\\
&[3^22^{2j}1^{2n-4j-5}],\text{ for }0\leq j\leq \textstyle\lfloor\frac{n-3}{2}\rfloor.
\end{split}
\end{equation}
Note that $\cNsm$ is the closure of the orbit $[3^2 2^{n-3}1]$ if $n$ is odd or $[3^22^{n-4}1^3]$ if $n$ is even. 

In this type we see some nontrivial Reeder pieces for the first time.
\begin{prop} \label{prop:B-orbits}
We have:
\begin{enumerate}
\item For all $0\leq j\leq \lfloor\frac{n}{2}\rfloor$, $\cM_{(1^{2j}0^{n-2j})}$ is a single $G$-orbit which $\pi$ maps isomorphically onto $[2^{2j}1^{2n-4j+1}]$.
\item $\cM_{(20^{n-1})}$ is a single $G$-orbit which $\pi$ maps isomorphically onto $[31^{2n-2}]$.
\item For all $1\leq j\leq \lfloor\frac{n-1}{2}\rfloor$, $\cM_{(21^{2j}0^{n-2j-1})}$ is the union of two $G$-orbits. One of them is $\cM_{(21^{2j}0^{n-2j-1})}^\iota$, and $\pi$ maps this isomorphically onto $[32^{2j}1^{2n-4j-2}]$. The other is mapped onto $[3^22^{2j-2}1^{2n-4j-1}]$ in a $2$-fold \'etale cover. In particular, the corresponding Reeder piece is the union of the two orbits $[32^{2j}1^{2n-4j-2}]$ and $[3^22^{2j-2}1^{2n-4j-1}]$.
\end{enumerate}
\end{prop}
\begin{proof}
Part (1) follows from Proposition \ref{prop:B-M}(1) and Lemma \ref{lem:laurent}(2). Now if $g\in\fG$ is such that $g\cdot\bo\in\Gr_{(21^{2j}0^{n-2j-1})}$, then as a point in the affine Grassmannian of $SL_{2n+1}$, $g\cdot\bo$ belongs to the orbit labelled by $(2,1,\ldots,1,0,\ldots,0,-1,\ldots,-1,-2)$ where there are $2j$ ones and $2n-4j-1$ zeroes. Lemma \ref{lem:laurent} tells us that $g=1+xt^{-1}+yt^{-2}$ where $\rk(y)=1$ and $\rk[\begin{smallmatrix}y&x\\0&y\end{smallmatrix}]=2j+2$. Of course, $x$ and $y$ are also constrained by Proposition \ref{prop:B-M}. If $g\cdot\bo$ is fixed by $\iota$, then by Proposition \ref{prop:B-M}(2), the rank conditions become $\rk(x^2)=1$ and $\rk(x)=2j+2$, equivalent to $x\in[32^{2j}1^{2n-4j-2}]$. If $g\cdot\bo$ is not fixed by $\iota$, then by Proposition \ref{prop:B-M}(3), the rank conditions become $\rk(x^2)=2$ and $\rk(x)=2j+2$, equivalent to $x\in[3^22^{2j-2}1^{2n-4j-1}]$ (in particular, this case is possible only when $j\geq 1$). It is then clear from Proposition \ref{prop:B-M}(2) that $\pi$ maps $\cM_{(21^{2j}0^{n-2j-1})}^\iota$ isomorphically onto $[32^{2j}1^{2n-4j-2}]$, and from Proposition \ref{prop:B-M}(3) that $\pi$ gives a $2$-fold \'etale covering map from $\cM_{(21^{2j}0^{n-2j-1})}\smallsetminus \cM_{(21^{2j}0^{n-2j-1})}^\iota$ to $[3^22^{2j-2}1^{2n-4j-1}]$. Since the domain of this covering map is open in $\cM_{(21^{2j}0^{n-2j-1})}$, it is irreducible and must be a single $G$-orbit.
\end{proof}

\begin{proof}[Proof of Proposition~\ref{prop:general} in type $B$]
The statements about the map $\pi$ are contained in Proposition \ref{prop:B-orbits}. All that remains are the statements about the Springer correspondence, but these were already checked by Reeder in \cite[Table 5.1]{reeder3}.
\end{proof}

\subsection{Type $D$}
\label{subsect:d}

In this subsection, let $G=Spin_{2n}$ for some integer $n\geq 4$. We make the usual identifications 
\[
\begin{split}
\Lambdav&=\{(a_1,\ldots,a_n)\in\Z^n\,|\,a_1+\cdots+a_n\in 2\Z\},\\
\Lambdav^+&=\{(a_1,\ldots,a_n)\in\Lambdav\,|\,a_1\geq \cdots\geq a_{n-1}\geq |a_n|\}.
\end{split}
\]
Under the map $G\to SL_{2n}$, with suitable choices of maximal tori and positive systems, $(a_1,\ldots,a_n)\in\Lambdav^+$ maps to the coweight $(a_1,\ldots,a_n,-a_n,\ldots,-a_1)$ for $SL_{2n}$. If $a_n<0$, the latter coweight is not dominant, but the dominant coweight in its Weyl group orbit is obtained simply by swapping the coordinates $a_n$ and $-a_n$.

\begin{lem} \label{lem:D-small}
We have 
\[
\begin{split}
\Lambdavsm^+&\textstyle=\{(21^{2j}0^{n-2j-1})\,|\,0\leq j\leq\lfloor\frac{n-1}{2}\rfloor\}\cup\{(1^{2j}0^{n-2j})\,|\,0\leq j\leq\lfloor\frac{n}{2}\rfloor\}\\
&\qquad\qquad\cup
\begin{cases}\{(1^{n-1}(-1))\},&\text{ if $n$ is even,}\\
\{(21^{n-2}(-1))\},&\text{ if $n$ is odd.}\end{cases}
\end{split}
\]
If $n$ is even, $\cM$ is irreducible, but if $n$ is odd, $\cM$ has two components which are interchanged by $\iota$.
\end{lem}
\begin{proof}
The proof of the description of $\Lambdavsm^+$ is identical to that given in Lemma \ref{lem:B-small}.

Note that the involution $\lambdav\mapsto-w_0\lambdav$ is nontrivial only when $n$ is odd, in which case it fixes every element of $\Lambdavsm^+$ except for interchanging $(21^{n-1})$ and $(21^{n-2}(-1))$. When $n$ is even, the partial order on $\Lambdavsm^+$ is described in the same way as the type-$B$ case, except that $(1^n)$ is replaced by two incomparable elements, $(1^{n})$ and $(1^{n-1}(-1))$; in particular, $(21^{n-2}0)$ is the unique maximal element, so $\cM$ is irreducible. When $n$ is odd, the partial order on $\Lambdavsm^+$ is described in the same way as the type-$B$ case, except that the maximal element $(21^{n-1})$ is replaced by two incomparable elements, $(21^{n-1})$ and $(21^{n-2}(-1))$, so $\cM$ has two irreducible components which are interchanged by $\iota$.
\end{proof}

As in the type-$B$ case, let $\cM'$ denote $\bigcup_{j}\cM_{(1^{2j}0^{n-2j})}$, or the union of this and $\cM_{(1^{n-1}(-1))}$ if $n$ is even.
\begin{prop} \label{prop:D-M}
The statement of Proposition \ref{prop:B-M} holds verbatim here.
\end{prop} 
\begin{proof}
Identical to that of Proposition \ref{prop:B-M}.
\end{proof}

As an immediate consequence, $\cNsm=\{x\in\cN\,|\,x^3=0,\rk(x^2)\leq 2\}$. It is well known that the $G$-orbits in $\cN$ are parametrized by their Jordan types, which are the partitions of $2n$ in which every even part has even multiplicity, except that there are two orbits (forming a single $O_{2n}$-orbit) for every partition in which no odd parts occur. The list of orbits belonging to $\cNsm$ is as follows:
\begin{equation} \label{eqn:D-orbitlist}
\begin{split}
&[2^{2j}1^{2n-4j}],\text{ for }0\leq j\leq \textstyle\lfloor\frac{n}{2}\rfloor,\\
&\qquad\text{except that there are two orbits $[2^n]_{I}$ and $[2^n]_{II}$ when $n$ is even,}\\
&[32^{2j}1^{2n-4j-3}],\text{ for }0\leq j\leq \textstyle\lfloor\frac{n-2}{2}\rfloor,\\
&[3^22^{2j}1^{2n-4j-6}],\text{ for }0\leq j\leq \textstyle\lfloor\frac{n-3}{2}\rfloor.
\end{split}
\end{equation}
Note that $\cNsm$ is the closure of the orbit $[3^2 2^{n-4}1^2]$ if $n$ is even or $[3^22^{n-3}]$ if $n$ is odd.

\begin{prop} \label{prop:D-orbits}
We have:
\begin{enumerate}
\item For all $0\leq j\leq \lfloor\frac{n-1}{2}\rfloor$, $\cM_{(1^{2j}0^{n-2j})}$ is a single $G$-orbit which $\pi$ maps isomorphically onto $[2^{2j}1^{2n-4j}]$. If $n$ is even, $\cM_{(1^n)}$ and $\cM_{(1^{n-1}(-1))}$ are single $G$-orbits which $\pi$ maps isomorphically onto $[2^n]_{I}$ and $[2^n]_{II}$ in some order.
\item $\cM_{(20^{n-1})}$ is a single $G$-orbit which $\pi$ maps isomorphically onto $[31^{2n-3}]$.
\item For all $1\leq j\leq \lfloor\frac{n-2}{2}\rfloor$, $\cM_{(21^{2j}0^{n-2j-1})}$ is the union of two $G$-orbits. One of them is $\cM_{(21^{2j}0^{n-2j-1})}^\iota$, and $\pi$ maps this isomorphically onto $[32^{2j}1^{2n-4j-3}]$. The other is mapped onto $[3^22^{2j-2}1^{2n-4j-2}]$ in a $2$-fold \'etale cover. In particular, the corresponding Reeder piece is the union of the two orbits $[32^{2j}1^{2n-4j-3}]$ and $[3^22^{2j-2}1^{2n-4j-2}]$.
\item If $n$ is odd, $\cM_{(21^{n-1})}$ and $\cM_{(21^{n-2}(-1))}$ are single $G$-orbits, each of which $\pi$ maps isomorphically onto $[3^22^{n-3}]$.
\end{enumerate}
\end{prop}
\begin{proof}
Similar to that of Proposition \ref{prop:B-orbits}.
\end{proof}

\begin{proof}[Proof of Proposition~\ref{prop:general} in type $D$]
The statements about the map $\pi$ are contained in Proposition \ref{prop:D-orbits}. All that remains are the statements about the Springer correspondence, but these follow from the computations of $V_\lambdav^\Tv$ done by Reeder in \cite[Lemma 3.2]{reeder1} and the description of the Springer correspondence in \cite[Section 13.3]{carter}. To illustrate, let $\lambdav=(21^{2j}0^{n-2j-1})$ for $1\leq j\leq \lfloor\frac{n-2}{2}\rfloor$. Then as a representation of $PSO_{2n}$, $V_\lambdav$ is what Reeder calls $V_{2j+1}$, and \cite[Lemma 3.2]{reeder1} says that the representation of $W$ on $V_\lambdav^\Tv$ is the sum of the irreducibles labelled by the bipartitions $((n-j-1,1);(j))$ and $((n-j-1);(j,1))$. After tensoring with sign, these become the irreducibles labelled by $((21^{n-j-2});(1^j))$ and $((21^{j-1});(1^{n-j-1}))$. These do indeed correspond under the Springer correspondence to the trivial and non-trivial local systems on the orbit $[3^22^{2j-2}1^{2n-4j-2}]$.
\end{proof}

\section{The exceptional types}
\label{sect:exceptional}

\subsection{Types $E_6$, $E_7$, $E_8$}
\label{subsect:e}

The poset $\Lambdavsm^+$ for each of these types is displayed in Table~\ref{tab:excpo}. Our numbering of the Dynkin diagrams of type $E$ follows \cite[Plates V--VII]{bourbaki}. Recall that in type $E_6$, the involution $\lambdav\mapsto-w_0\lambdav$ interchanges $\omegav_1$ and $\omegav_6$, as well as $\omegav_3$ and $\omegav_5$. In types $E_7$ and $E_8$, $\lambdav=-w_0\lambdav$ for all $\lambdav\in\Lambdav$. Inspecting the poset $\Lambdavsm^+$, we see:
\begin{lem} \label{lem:E-Mirred}
In type $E_6$, $\cM$ has two irreducible components which are interchanged by $\iota$. In types $E_7$ and $E_8$, $\cM$ is irreducible. \qed
\end{lem}

We will study the map $\pi$ by means of a large simple subgroup $H \subset G$ of classical type for which the results of the previous section are available.  We define $H$ by specifying a connected sub-diagram of the extended Dynkin diagram of $G$, namely the one whose nodes have the following labels, with $0$ denoting the added node; we have listed the nodes in the order appropriate to the type of the sub-diagram.
\[
\begin{array}{c|c}
G & H \\
\hline
E_6 & 1,3,4,5,6 \text{ (type $A_5$)} \\
E_7 & 0,1,3,4,2,5 \text{ (type $D_6$)}\\
E_8 & 0,8,7,6,5,4,3,2 \text{ (type $D_8$)}
\end{array}
\]
This sub-diagram generates a subsystem $\Psi$ of the root system of $G$, and we let $H$ be the subgroup generated by the corresponding root subgroups.  The intersection $T\cap H$ is a maximal torus of $H$, and its cocharacter lattice $\Lambdav_H=\Q\check{\Psi}\cap\Lambdav$ is a sub-lattice of $\Lambdav$. Since $H$ is of classical type, we can write elements of $\Lambdav_H$ as tuples of integers as in Section~\ref{sect:classical}, and use the description of $\Lambdav_{H,\sm}^+$ given there. (Note that if $G$ is of type $E_8$, the coroot lattice $\Z\check{\Psi}$ of $H$ is an index-$2$ subgroup of $\Lambdav_H$, so $H$ is not simply connected; recall that 
$\Lambdav_{H,\sm}^+$ is still contained in the coroot lattice, by definition.) 

To begin, we carry out some computations according to the following procedure.  The results of these computations are recorded in Table~\ref{tab:Ecalc}.

\begin{enumerate}
\item \textit{For each $\lambdav\in\Lambdavsm^+$, compute $\dim\Gr_\lambdav$.} This is done using the formula
$\dim\Gr_\lambdav=\langle\lambdav,2\rho\rangle$ where $2\rho$ is the sum of the positive roots of $G$.

\item \textit{For each $\lambdav\in\Lambdavsm^+$, find the elements $\muv\in\Lambdav_{H,\sm}^+$ such that the $\HO$-orbit $\Gr_{H,\muv}$ is contained in $\Gr_\lambdav$.} For $\muv\in\Lambdav_H^+$, we have $\Gr_{H,\muv} \subset \Gr_\lambdav$ if and only if $\muv$, when regarded as an element of $\Lambdav$, lies in the $W$-orbit of $\lambdav$.  The determination of which $\muv\in\Lambdav_{H,\sm}^+$ have this property was carried out using the \textsf{LiE} software package~\cite{lie}.

\item \textit{For each such $\muv\in\Lambdav_{H,\sm}^+$, find the nilpotent orbits in $\cN_{H,\sm}$ 
contained in $\pi_H(\cM_{H,\muv})$.}  These orbits can be found by referring to Table~\ref{tab:classcalc}.

\item \textit{For each such nilpotent orbit in $\cN_{H,\sm}$, compute its $G$-saturation in $\cN$.}  Recall that in the Bala--Carter classification of nilpotent orbits, a nilpotent orbit is labelled by the smallest Levi subalgebra it meets.  In the classical types, the procedure for converting from a partition-type label to a Bala--Carter label is given in~\cite[\S 6]{bc2}.  When we do this for a nilpotent orbit $C \subset \cN_{H,\sm}$ found in the previous step, we notice that in each case, the Levi subalgebra of $\fh$ that arises shares its derived subalgebra with a Levi subalgebra of $\fg$.  Therefore, the $G$-saturation of $C$ is the orbit in $\cN$ which carries the same Bala--Carter label.  These labels are recorded in the last column, along with the dimension of the orbit as given in \cite[Section 13.1]{carter}.  (In some instances, two isomorphic but nonconjugate Levi subalgebras of $\fh$ may become conjugate under $G$.  For instance, in $D_8$, the nilpotent orbits $[3 2^4 1^5]$ and $[2^8]$ are labelled by two conjugate subalgebras that are both of type $4A_1$.)
\end{enumerate}

In view of Lemma~\ref{lem:commute}, we can conclude that each nilpotent orbit listed in the right-hand column of Table~\ref{tab:Ecalc} in the row corresponding to $\lambdav\in\Lambdavsm^+$ is contained in the image $\pi(\cM_\lambdav)$. We now aim to prove that this list of orbits is complete.

\begin{table}
\[\small
\begin{array}{@{}cc@{}}
E_6: &
\begin{array}{cc|c|c|cc}
\lambdav & \dim \Gr_\lambdav &
  \muv : \Gr_{H,\muv} \subset \Gr_\lambdav &
  \pi_H(\cM_{H,\muv}) &
  G \cdot \pi_H(\cM_{H,\muv}) & \dim  \\
\hline
3\omegav_1 & 48 &
   (1,1,1,1,-2,-2) & [3^2] & 2A_2 & 48 \\
\hline
3\omegav_6 & 48 &
   (2,2,-1,-1,-1,-1) & [3^2] & 2A_2 & 48 \\
\hline
\omegav_1 + \omegav_3 & 46 &
   (1,1,1,0,-1,-2) & [3 2 1] & A_2+A_1 & 46 \\
\hline
\omegav_5 + \omegav_6 & 46 &
   (2,1,0,-1,-1,-1) & [3 2 1] & A_2+A_1 & 46 \\
\hline
\omegav_4 & 42 &
   (1,1,0,0,0,-2) & [3 1^3] & A_2 & 42 \\
   \cline{3-6}
 &&(2,0,0,0,-1,-1) & [3 1^3] & A_2 & 42 \\
   \cline{3-6}
 &&(1,1,1,-1,-1,-1) & [2^3] & 3A_1 & 40 \\
\hline
\omegav_1 + \omegav_6 & 32 &
   (1,1,0,0,-1,-1) & [2^2 1^2] & 2A_1 & 32 \\
\hline
\omegav_2 & 22 &
   (1,0,0,0,0,-1) & [2 1^4] & A_1 & 22 \\
\hline
0 & 0 & 0 & [1^6] & 1 & 0
\end{array}
\\
\ \\ \ \\
E_7: &
\begin{array}{cc|c|c|cc}
\lambdav & \dim \Gr_\lambdav &
  \muv : \Gr_{H,\muv} \subset \Gr_\lambdav &
  \pi_H(\cM_{H,\muv}) &
  G \cdot \pi_H(\cM_{H,\muv}) & \dim  \\
\hline
\omegav_2+\omegav_7 & 76 &
   (2,1,1,1,1,0) & [3^2 2^2 1^2] & A_2+A_1 & 76 \\
   &&&             [3 2^4 1] & 4A_1 & 70 \\
\hline
\omegav_3 & 66 &
   (2,1,1,0,0,0) & [3^2 1^6] & A_2 & 66 \\
   &&&             [3 2^2 1^5] & (3A_1)' & 64 \\
   \cline{3-6}
 &&(1,1,1,1,1,1) & [2^6]_{I} & (3A_1)' & 64 \\
\hline
2\omegav_7 & 54 &
   (1,1,1,1,1,-1) & [2^6]_{II} & (3A_1)'' & 54 \\
\hline
\omegav_6 & 52 &
   (2,0,0,0,0,0) & [3 1^9] & 2A_1 & 52 \\
   \cline{3-6}
 &&(1,1,1,1,0,0) & [2^4 1^4] & 2A_1 & 52 \\
\hline
\omegav_1 & 34 &
   (1,1,0,0,0,0) & [2^2 1^8] & A_1 & 34 \\
\hline
0 & 0 & 0 & [1^{12}] & 1 & 0
\end{array}
\\
\ \\ \ \\
E_8: &
\begin{array}{cc|c|c|cc}
\lambdav & \dim \Gr_\lambdav &
  \muv : \Gr_{H,\muv} \subset \Gr_\lambdav &
  \pi_H(\cM_{H,\muv}) &
  G \cdot \pi_H(\cM_{H,\muv}) & \dim  \\
\hline
\omegav_2 & 136 &
   (2,1,1,1,1,0,0,0) & [3^2 2^2 1^6] & A_2+A_1 & 136 \\
   &&&                 [3 2^4 1^5] & 4A_1 & 128 \\
   \cline{3-6}
 &&(1,1,1,1,1,1,1,1) & [2^8] & 4A_1 & 128 \\
\hline
\omegav_7 & 114 &
   (2,1,1,0,0,0,0,0) & [3^2 1^{10}] & A_2 & 114 \\
   &&&                 [3 2^2 1^9] & 3A_1 & 112 \\
   \cline{3-6}
 &&(1,1,1,1,1,1,0,0) & [2^6 1^4] & 3A_1 & 112 \\
\hline
\omegav_1 & 92 &
   (2,0,0,0,0,0,0,0) & [3 1^{13}] & 2A_1 & 92 \\
   \cline{3-6}
 &&(1,1,1,1,0,0,0,0) & [2^4 1^8] & 2A_1 & 92 \\
\hline
\omegav_8 & 58 &
   (1,1,0,0,0,0,0,0) & [2^2 1^{12}] & A_1 & 58 \\
\hline
0 & 0 & 0 & [1^{16}] & 1 & 0
\end{array} \\
\ 
\end{array}
\]
\caption{Orbit calculations for types $E_6$, $E_7$, $E_8$.}\label{tab:Ecalc}
\end{table}

\begin{table}
\[\small
\begin{array}{c}
\multicolumn{1}{l}{E_6:} \\
\begin{array}{c|ccc|ccc|l}
 & V_{3\omegav_1} & & & V_{3\omegav_6} &
 & & \IC(2A_2) \\
\cline{1-2}\cline{4-5}\cline{7-8}
3\omegav_1 & 1 & & 3\omegav_6 & 1 &
 & 2A_2 & q^{12} \\
\cline{1-2}\cline{4-5}\cline{7-8}
\omegav_1 + \omegav_3 & 1 & & \omegav_5 + \omegav_6 & 1 &
  & A_2 + A_1 & q^{12} \\
\cline{1-2}\cline{4-5}\cline{7-8}
\omegav_4 & 1 & & \omegav_4 & 1 &
  & A_2 & 2q^{12} \\
 &&&&&& 3A_1 & q^{12} \\
\cline{1-2}\cline{4-5}\cline{7-8}
\omegav_1 + \omegav_6 & 4 & & \omegav_1 + \omegav_6 & 4 &
  & 2A_1 & q^{18} + q^{16} + q^{14} + q^{12} \\
\cline{1-2}\cline{4-5}\cline{7-8}
\omegav_2 & 10 & & \omegav_2 & 10 &
  & A_1 & 
  \tiny\begin{array}{@{}l@{}}
  q^{21} + q^{20} + q^{19} + 2q^{18} + q^{17} \\ {}+ q^{16} + q^{15} + q^{14} + q^{12}
  \end{array} \\
\cline{1-2}\cline{4-5}\cline{7-8}
0 & 24 & & 0 & 24 &
  & 1 & 
  \tiny\begin{array}{@{}l@{}}
  q^{30} + q^{28} + q^{27} + q^{26} + q^{25} + 3q^{24} + q^{23}\\{} + 2q^{22} + 2q^{21} + 2q^{20} + q^{19} + 3q^{18}\\{} + q^{17} + q^{16} + q^{15} + q^{14} + q^{12}
  \end{array}
\end{array}
\\
\ \\
\multicolumn{1}{l}{E_7:} \\
\begin{array}{c|ccc|ll}
  & V_{\omegav_2+\omegav_7} &
  & & \IC(A_2+A_1) & \IC(A_2+A_1,\sigma) \\
\cline{1-2}\cline{4-6}
\omegav_2+\omegav_7 & 1 &
  & A_2+A_1 & q^{25} & q^{25} \\
 &&& 4A_1 & q^{25} & \\
\cline{1-2}\cline{4-6}
\omegav_3 & 5 &
 & A_2 & 
  \tiny\begin{array}{@{}l@{}}
  q^{29}+q^{28}+q^{27}\\{}+q^{26}+q^{25}
  \end{array} &
  \tiny\begin{array}{@{}l@{}}
  q^{29}+q^{28}+q^{27}\\{}+q^{26}+q^{25}
  \end{array} \\
 &&& (3A_1)' & q^{29}+q^{27}+q^{25} & q^{28}+q^{26} \\
\cline{1-2}\cline{4-6}
2\omegav_7 & 6 &
 & (3A_1)'' & q^{35}+q^{31}+q^{29}+q^{25} & q^{32}+q^{28} \\
\cline{1-2}\cline{4-6}
\omegav_6 & 22 &
 & 2A_1 &
  \tiny\begin{array}{@{}l@{}}
  2q^{35}+2q^{33}+3q^{31}\\{}+3q^{29}+q^{27}+q^{25}
  \end{array} &
  \tiny\begin{array}{@{}l@{}}
  q^{36}+q^{34}+3q^{32}\\{}+2q^{30}+2q^{28}+q^{26}
  \end{array} \\
\cline{1-2}\cline{4-6}
\omegav_1 & 75 &
  & A_1 &
  \tiny\begin{array}{@{}l@{}}
  2q^{43}+3q^{41}+5q^{39}+7q^{37}\\{}+7q^{35}+6q^{33}+5q^{31}\\{}+3q^{29}+q^{27}+q^{25}
  \end{array} &
  \tiny\begin{array}{@{}l@{}}
  q^{44}+2q^{42}+4q^{40}+5q^{38}\\{}+6q^{36}+6q^{34}+5q^{32}\\{}+3q^{30}+2q^{28}+q^{26}
  \end{array} \\
\cline{1-2}\cline{4-6}
0 & 225 &
  & 1 &  
  \tiny\begin{array}{@{}l@{}}
q^{59}+q^{57}+3q^{55}+5q^{53}\\{}+6q^{51}+9q^{49}+11q^{47}\\{}+11q^{45}+13q^{43}+13q^{41}\\{}+11q^{39}+11q^{37}+9q^{35}\\{}+6q^{33}+5q^{31}+3q^{29}\\{}+q^{27}+q^{25}
  \end{array} &
  \tiny\begin{array}{@{}l@{}}
q^{58}+2q^{56}+3q^{54}+5q^{52}\\{}+7q^{50}+8q^{48}+10q^{46}\\{}+11q^{44}+11q^{42}+11q^{40}\\{}+10q^{38}+8q^{36}+7q^{34}\\{}+5q^{32}+3q^{30}\\{}+2q^{28}+q^{26}
  \end{array} 
\end{array}
\\
\ \\
\multicolumn{1}{l}{E_8:} \\
\begin{array}{@{}c|ccc|ll@{}}
  & V_{\omegav_2} &
  & & \IC(A_2+A_1) & \IC(A_2+A_1,\sigma) \\
\cline{1-2}\cline{4-6}
\omegav_2 & 1 &
  & A_2+A_1 & q^{52} & q^{52} \\
 &&& 4A_1 & q^{52} & \\
\cline{1-2}\cline{4-6}
\omegav_7 & 6 &
  & A_2 & 
  \tiny\begin{array}{@{}l@{}}
  q^{62}+q^{59}+q^{58}+q^{56}\\{}+q^{55}+q^{52}
  \end{array} &
  \tiny\begin{array}{@{}l@{}}
  q^{62}+q^{59}+q^{58}+q^{56}\\{}+q^{55}+q^{52}
  \end{array} \\
 &&& 3A_1 & q^{62}+q^{58}+q^{56}+q^{52} & q^{59}+q^{55} \\
\cline{1-2}\cline{4-6}
\omegav_1 & 29 &
  & 2A_1 &
  \tiny\begin{array}{@{}l@{}}
  q^{72}+q^{70}+2q^{68}+2q^{66}+3q^{64}\\{}+2q^{62}+2q^{60}+2q^{58}+q^{56}+q^{52}
  \end{array} &
  \tiny\begin{array}{@{}l@{}}
  q^{73}+q^{69}+2q^{67}+q^{65}+2q^{63}\\{}+2q^{61}+q^{59}+q^{57}+q^{55}
  \end{array} \\
\cline{1-2}\cline{4-6}
\omegav_8 & 111 &
  & A_1 &
  \tiny\begin{array}{@{}l@{}}
  q^{88}+q^{86}+2q^{84}+3q^{82}+4q^{80}\\{}+4q^{78}+6q^{76}+6q^{74}+6q^{72}+6q^{70}\\{}+6q^{68}+4q^{66}+5q^{64}+3q^{62}+2q^{60}\\{}+2q^{58}+q^{56}+q^{52}
  \end{array} &
  \tiny\begin{array}{@{}l@{}}
  q^{89}+2q^{85}+2q^{83}+2q^{81}+4q^{79}\\{}+4q^{77}+3q^{75}+6q^{73}+4q^{71}\\{}+4q^{69}+5q^{67}+3q^{65}+2q^{63}\\{}+3q^{61}+q^{59}+q^{57}+q^{55}
  \end{array} \\
\cline{1-2}\cline{4-6}
0 & 370 &
  & 1 &
  \tiny\begin{array}{@{}l@{}}
  q^{116}+q^{112}+2q^{110}+2q^{108}+3q^{106}\\{}+5q^{104}+4q^{102}+7q^{100}+8q^{98}\\{}+8q^{96}+10q^{94}+12q^{92}+10q^{90}\\{}+13q^{88}+13q^{86}+12q^{84}+13q^{82}\\{}+13q^{80}+10q^{78}+12q^{76}+10q^{74}\\{}+8q^{72}+8q^{70}+7q^{68}+4q^{66}\\{}+5q^{64}+3q^{62}+2q^{60}\\{}+2q^{58}+q^{56}+q^{52}
  \end{array} &
  \tiny\begin{array}{@{}l@{}}
  q^{113}+q^{111}+q^{109}+3q^{107}\\{}+2q^{105}+3q^{103}+6q^{101}+4q^{99}\\{}+6q^{97}+9q^{95}+6q^{93}+9q^{91}\\{}+11q^{89}+7q^{87}+11q^{85}+11q^{83}\\{}+7q^{81}+11q^{79}+9q^{77}+6q^{75}\\{}+9q^{73}+6q^{71}+4q^{69}+6q^{67}\\{}+3q^{65}+2q^{63}+3q^{61}\\{}+q^{59}+q^{57}+q^{55}
  \end{array} 
\end{array}
\\ \ 
\end{array}
\]
\caption{Calculations for the proof of Lemma~\ref{lem:E-stalkcalc}.}\label{tab:Estalk}
\end{table}

\begin{lem}\label{lem:E-openfibre}
Let $C \subset \cN$ be the unique maximal $G$-orbit appearing in Table~\ref{tab:Ecalc}, and let $D = \pi^{-1}(C) \subset \cM$.  Then $D$ is an open dense subset of $\cM$, and $\pi|_D: D \to C$ is an \'etale double cover.  In fact, in types $E_7$ and $E_8$, $D$ is isomorphic to the unique connected $G$-equivariant double cover of $C$, whereas in type $E_6$, $D$ is isomorphic to the trivial double cover $\Z/2\Z \times C$.
\end{lem}
\begin{proof}
We see by inspection that in each case $\dim C = \dim \cM$, so we must have $\dim C = \dim D$.  
Since $\pi$ is $G$-equivariant, it follows immediately that $\pi|_D: D\to C$ is finite and \'etale.  In types $E_7$ and $E_8$, $\cM$ is irreducible, so $D$ must be a connected dense open subset of $\cM$.  Note that for a point $x \in C \cap \cN_{H,\sm}$, the fibre $\pi_H^{-1}(x)$ has two points.  So for general $x\in C$, the fibre $\pi^{-1}(x)$ must have at least two points.  But in both these types, the $G$-equivariant fundamental group of $C$ is $\Z/2\Z$ (see
\cite[Section 13.1]{carter}), so the fibres of a connected cover of $C$ can have at most two points.  We conclude that $D$ is the unique connected double cover of $C$ in these types.

Now suppose that $G$ is of type $E_6$. Since $D$ is $\iota$-stable, it must meet both irreducible components of $\cM$. It follows that $D$ is dense in $\cM$ and has two connected components.  Since the $G$-equivariant fundamental group of $C$ is trivial in this case, each connected component of $D$ must be isomorphic to $C$.
\end{proof}

\begin{lem}\label{lem:E-cnsm}
The image $\cNsm=\pi(\cM)$ is an irreducible closed subset of $\cN$, containing precisely the nilpotent orbits appearing in Table~\ref{tab:Ecalc}. In particular, $\cNsm\subset G\cdot\cN_{H,\sm}$.
\end{lem}
\begin{proof}
Let $C$ and $D$ be as in Lemma~\ref{lem:E-openfibre}.  Since $D$ is dense in $\cM$ and $\pi(D) = C$, it follows that $\pi(\cM) \subset \overline{C}$.  On the other hand, consulting \cite[Section 13.4]{carter} for the closure
order, we see that every $G$-orbit in $\overline{C}$ appears in Table~\ref{tab:Ecalc}, and so is contained in $\pi(\cM)$.  Thus, $\pi(\cM) = \overline{C}$.
\end{proof}

\begin{lem}\label{lem:pi-finite}
The map $\pi: \cM \to \cNsm$ is finite.
\end{lem}
\begin{proof}
This follows from Lemma~\ref{lem:E-openfibre} and Lemma~\ref{lem:finite-map}.
\end{proof}

\begin{lem}\label{lem:E-stalkdim}
Let $\lambdav\in\Lambdavsm^+$.  For $\muv \in \Lambdav$, let $m_\lambdav^\muv$ denote the dimension of the $\muv$-weight space in $V_\lambdav$.  For $x \in \cNsm$, we have
\[
\sum_{i} \dim \cH_x^i(\pi_*\IC(\overline{\Gr_\lambdav})|_\cM)
= \sum_{\muv\in\Lambdavsm^+}
 | \pi^{-1}(x) \cap \Gr_\muv|\, m_\lambdav^\muv.
\]
\end{lem}
\begin{proof}
Since $\cM$ is open in $\Grsm$ and $\pi: \cM \to \cNsm$ is finite, we have
\[
\cH_x^i(\pi_*\IC(\overline{\Gr_\lambdav})|_\cM) \cong
\bigoplus_{y \in \pi^{-1}(x)} \cH^i_y(\IC(\overline{\Gr_\lambdav})).
\]
Of course, since $\IC(\overline{\Gr_\lambda})$ is $\GO$-equivariant, its stalk at a point $y$ depends, up to isomorphism, only on the $\GO$-orbit, so we have
\[
\dim \cH_x^i(\pi_*\IC(\overline{\Gr_\lambdav})|_\cM) =
\sum_{\muv \in \Lambdavsm^+} |\pi^{-1}(x) \cap \Gr_\muv|\, \dim \cH^i_\muv(\IC(\overline{\Gr_\lambdav})),
\]
where $\cH^i_\muv(\IC(\overline{\Gr_\lambdav}))$ denotes the stalk of $\cH^i(\IC(\overline{\Gr_\lambdav}))$ at some chosen point of $\Gr_\muv$.  Now $\sum_{i} \dim \cH^i_\muv(\IC(\overline{\Gr_\lambdav})) q^{i/2}$ is essentially Lusztig's $q$-analogue of the weight multiplicity.  In fact, it follows from~\cite{lusztig:scf} that $\sum_{i} \dim \cH^i_\muv(\IC(\overline{\Gr_\lambdav})) = m_\lambdav^\muv$, and the lemma follows from that.
\end{proof}

\begin{lem}\label{lem:E-stalkcalc}
Let $C \subset \cNsm$ be the unique open orbit, and let $D_1$ be a connected component of $\pi^{-1}(C)$.  
Let $\lambdav\in\Lambdavsm^+$ be such that $D_1 \subset \Gr_\lambdav$.  For $x \in \cNsm\cap\cN_{H,\sm}$, we have
\begin{equation}\label{eqn:E-stalkcalc}
\sum_{i} \dim \cH_x^i(\pi_*\IC(\overline{D_1}))
= \sum_{\muv\in\Lambdavsm^+}
 | \pi_H^{-1}(x) \cap \Gr_\muv|\, m_\lambdav^\muv.
\end{equation}
\end{lem}
\begin{proof}
Each side of this formula depends only on the $H$-orbit of $x$.  Let $Y_H$ denote this orbit; it must be one of those appearing in Table~\ref{tab:Ecalc}.  The proof consists of simply calculating both sides separately for each possible orbit, and checking that the calculations agree.

We explain first how to calculate the right-hand side.  Because Proposition~\ref{prop:general} holds for $H$, we know the cardinality of $\pi_H^{-1}(x)$ (it is either~$1$ or~$2$).  That same proposition also tells us the $\HO$-orbits to which these points belong; then, by referring to Table~\ref{tab:Ecalc}, one can determine the $\GO$-orbit containing each point of $\pi_H^{-1}(x)$.  Finally, the multiplicities $m_\lambdav^\muv$ are known from (say) the Freudenthal multiplicity formula.  For explicit calculations, the authors relied on the LiE software package~\cite{lie}.

For the left-hand side, note that because $\pi$ is finite, the functor $\pi_*$ is $t$-exact for perverse sheaves, and it takes intersection cohomology complexes to intersection cohomology complexes.  Indeed, it follows from Lemma~\ref{lem:E-openfibre} that
\[
\pi_*\IC(\overline{D_1}) \cong
\begin{cases}
\IC(\overline C) \oplus \IC(\overline C, \sigma) & \text{in types $E_7$ and $E_8$,} \\
\IC(\overline C) & \text{in type $E_6$.}
\end{cases}
\]
Here, $\sigma$ denotes the unique nontrivial local system on $C$ in types $E_7$ and $E_8$.  The stalks of simple perverse sheaves on $\cN$ can be computed by the so-called Lusztig--Shoji algorithm (see~\cite[\S 24]{lusztig:cs5} or~\cite[\S 4]{shoji}), for which an implementation for the GAP computer algebra system is available from~\cite{a:lsa}.  For each simple perverse sheaf $\cF$ and each $x$, this algorithm computes the polynomial
\[
q^{(\dim \cN)/2}\sum_{i} \dim \cH_x^i(\cF)\, q^{i/2}.
\]
The relevant polynomials (which depend only on the $G$-orbit of $x$, of course) are recorded in Table~\ref{tab:Estalk}.  Evaluating these polynomial at $q = 1$ yields the left-hand side of~\eqref{eqn:E-stalkcalc}.  We leave it to the reader to compare the left- and right-hand sides of~\eqref{eqn:E-stalkcalc} in each case.
\end{proof}

\begin{cor}\label{cor:E-fibre}
For $x \in \cNsm\cap\cN_{H,\sm}$, we have $\pi^{-1}(x) = \pi_H^{-1}(x)$.
\end{cor}
\begin{proof}
Retain the notation in the statement of Lemma~\ref{lem:E-stalkcalc}.  Comparing that statement to Lemma~\ref{lem:E-stalkdim}, we see that
\[
\sum_{\muv\in\Lambdavsm^+}
|(\pi^{-1}(x) \smallsetminus \pi_H^{-1}(x)) \cap \Gr_\muv|\, m_\lambdav^\muv = 0.
\]
In types $E_7$ and $E_8$, we have $m_\lambdav^\muv > 0$ for all $\muv\in\Lambdavsm^+$.  In type $E_6$, $\Grsm$ has two components, and there are two choices for $D_1$ and for $\lambdav$ in Lemma~\ref{lem:E-stalkcalc}.  For each $\muv$, we have $m_\lambdav^\muv > 0$ for at least one of the two choices of $\lambdav$.  In all three types, we may then conclude that $\pi^{-1}(x) \smallsetminus \pi_H^{-1}(x) = \varnothing$, as desired.
\end{proof}

\begin{cor}\label{cor:E-meeting}
Every $G$-orbit in $\cM$ meets $\cM_H$.
\end{cor}
\begin{proof}
Suppose $D'$ is a $G$-orbit in $\cM$ that does not meet $\cM_H$. We know from Lemma~\ref{lem:E-cnsm}
that $\pi(D')$ is a $G$-orbit in $G\cdot\cN_{H,\sm}$. But then for $x\in\pi(D')\cap\cN_{H,\sm}$,
the fibre $\pi^{-1}(x)$ contains a point of $D'$ which is not in $\pi_H^{-1}(x)$, contradicting
Corollary~\ref{cor:E-fibre}.  
\end{proof}

\begin{table}
\[
E_6:
\begin{array}{c|c}
\lambdav & V_\lambdav^\Tv \otimes \epsilon \\
\hline
3\omegav_1, 3\omegav_6 & \phi_{24,12} \\
\omegav_1 + \omegav_3, \omegav_5 + \omegav_6 & \phi_{64,13} \\
\omegav_4 & \phi_{30,15} + \phi_{15,17} \\
\omegav_1 + \omegav_6 & \phi_{20,20} \\
\omegav_2 & \phi_{6,25} \\
0 & \phi_{1,36}
\end{array}
\qquad
E_7:
\begin{array}{c|c}
\lambdav & V_\lambdav^\Tv \otimes \epsilon \\
\hline
\omegav_2+\omegav_7 & \phi_{120,25} + \phi_{105,26} \\
\omegav_3 & \phi_{56,30} + \phi_{21,33} \\
2\omegav_7 & \phi_{21,36} \\
\omegav_6 & \phi_{27,37} \\
\omegav_1 & \phi_{7,46} \\
0 & \phi_{1,63}
\end{array}
\]
\bigskip
\[
E_8: 
\begin{array}{c|c}
\lambdav & V_\lambdav^\Tv \otimes \epsilon \\
\hline
\omegav_2 & \phi_{210,52} + \phi_{160,55}\\
\omegav_7 & \phi_{112,63} + \phi_{28,68} \\
\omegav_1 & \phi_{35,74} \\
\omegav_8 & \phi_{8,91} \\
0 & \phi_{1,120}
\end{array}
\qquad
F_4:
\begin{array}{c|c}
\lambdav & V_\lambdav^\Tv \otimes \epsilon \\
\hline
\omegav_2 & \phi_{8,9}'' + \phi_{1,12}'' \\
\omegav_4 & \phi_{4,13} \\
\omegav_1 & \phi_{2,16}'' \\
0 & \phi_{1,24}
\end{array}
\qquad
G_2:
\begin{array}{c|c}
\lambdav & V_\lambdav^\Tv \otimes \epsilon \\
\hline
\omegav_1 & \phi_{2,1} \\
\omegav_2 & \phi_{1,3}'' \\
0 & \phi_{1,6} 
\end{array}
\]
\bigskip
\caption{Zero weight spaces in the exceptional groups}\label{tab:exczw}
\end{table}

We are now ready to prove the main result of this section.

\begin{proof}[Proof of Proposition~\ref{prop:general} for types $E_6$, $E_7$, and $E_8$]
The first two statements are in Lemmas~\ref{lem:E-Mirred},~\ref{lem:E-openfibre} and~\ref{lem:E-cnsm}.  Now, let $\lambdav\in\Lambdavsm^+$. By Corollary~\ref{cor:E-meeting}, the nilpotent orbits in $\cNsm$ listed
for $\lambdav$ in Table~\ref{tab:Ecalc} are precisely those in the Reeder piece $\pi(\cM_{\lambdav})$.
We may check by inspection that for distinct $\lambdav,\nuv\in\Lambdavsm^+$, the Reeder pieces 
$\pi(\cM_\lambdav)$ and $\pi(\cM_\nuv)$ are either disjoint or equal, and that equality occurs if and only if $\nuv = -w_0\lambdav$.  Thus, we have established the bijection~\eqref{eqn:gen-bij} and the fact that the Reeder pieces form a partition of $\cNsm$.

It is also clear by inspection that each Reeder piece consists of one or two nilpotent orbits.  Consider now a $\iota$-stable union of $\GO$-orbits $\Gr_\lambdav \cup \Gr_{-w_0\lambdav}$ contained in $\Grsm$, and let $S$ be the corresponding Reeder piece.  In cases where $S$ consists of a single nilpotent orbit $C$, we see from Table~\ref{tab:Ecalc} that $\pi_H^{-1}(x) \cap \Gr_\lambdav$ is a singleton for all $x \in C \cap \cN_{H,\sm}$.   By Corollary~\ref{cor:E-fibre} and the $G$-equivariance of $\pi$, it follows that $\pi^{-1}(x) \cap \Gr_\lambdav$ is a singleton for all $x \in C$, so in fact, $\pi$ gives rise to an isomorphism $\pi^{-1}(C) \cap \Gr_\lambdav \to C$.

If $S$ consists of two nilpotent orbits $C_1$ and $C_2$ with $C_2 \subset \overline{C_1}$, then we see by inspection that $\Gr_\lambdav = \Gr_{-w_0\lambdav}$.  The reasoning of the previous paragraph applies verbatim to $C_2$.  Similar reasoning shows that $\pi^{-1}(C_1) \to C_1$ is a $2$-fold \'etale cover.  Moreover, $\pi^{-1}(C_1)$ must be connected because it has the same dimension as $\Gr_\lambdav$, and the latter is irreducible.  Finally, using the fact that Proposition~\ref{prop:general} holds for $H$, we see that the $\Z/2\Z$-action is free on fibres over points of $C_1 \cap \cN_{H,\sm}$, and therefore on all of $\pi^{-1}(C_1)$.

We have now established all the geometric assertions in Proposition~\ref{prop:general}.  It remains to check the claims involving the Springer correspondence.  For each small representation $V$ in type $E$, the Weyl group action on $V^\Tv$ has been computed by Reeder~\cite[\S 4]{reeder1}.  In Table~\ref{tab:exczw}, we record the results of tensoring Reeder's calculations with $\epsilon$.  Finally, one may consult the tables for the Springer correspondence 
in~\cite[Section 13.3]{carter} to verify that either~\eqref{eqn:zspring-single} or~\eqref{eqn:zspring-double} holds, as appropriate.
\end{proof}

\begin{table}
\tikzstyle{reed1}=[draw,inner sep=0pt,shape=ellipse,minimum height=.75cm, minimum width=0.75cm]
\tikzstyle{reed2}=[draw,shape=ellipse,rotate=-40,minimum height=1cm,minimum width=3cm]
\tikzstyle{lreed2}=[draw,shape=ellipse,rotate=38,minimum height=1cm,minimum width=3.2cm]
\tikzstyle{pimap}=[dashed,->]
\begin{center}
\small
\begin{tabular}{cc}
\begin{tikzpicture}[x=.9cm]
  \path (-4,0)  node (w13)  {$3\omegav_1$};
  \path (-2,0)  node (w63)  {$3\omegav_6$};
  \path (-4,-1) node (w1p3) {$\omegav_1 + \omegav_3$};
  \path (-2,-1) node (w5p6) {$\omegav_5 + \omegav_6$};
  \path (-3.0,-2) node (w4) {$\omegav_4$};
  \path (-3.0,-4) node (w1p6) {$\omegav_1 + \omegav_6$};
  \path (-3.0,-5) node (w2) {$\omegav_2$};
  \path (-3.0,-6) node (w0) {$0$};

  \path (0,0)   node[reed1] (a22)  {$2A_2$};
  \path (0,-1)  node[reed1] (a2a1) {$A_2+A_1$};
  \path (0,-2)  node (a2)   {$A_2$};
  \path (1,-3)  node (a13)  {$3A_1$};
  \path (0,-4)  node[reed1] (a12)  {$2A_1$};
  \path (0,-5)  node[reed1] (a1)   {$A_1$};
  \path (0,-6)  node[reed1] (a0)   {$1$};
  
  \path (0.5,-2.5) coordinate[reed2] (e1);
  
  \draw (a22) -- (a2a1) -- (a2) -- (a13) -- (a12) -- (a1) -- (a0)
        (a2) -- (a12);
  \draw (w13) -- (w1p3) -- (w4) -- (w1p6) -- (w2) -- (w0)
        (w63) -- (w5p6) -- (w4);
  \draw[pimap] (w63) .. controls(-1,-0.25) .. (a22);
  \draw[pimap] (w13) .. controls(-2,0.5) .. (a22);
  \draw[pimap] (w5p6) .. controls(-1,-1.25) .. (a2a1);
  \draw[pimap] (w1p3) .. controls(-2,-0.5) .. (a2a1);
  \draw[pimap] (w4) .. controls(-1.5,-2) .. (e1);
  \draw[pimap] (w1p6) -- (a12);
  \draw[pimap] (w2) -- (a1);
  \draw[pimap] (w0) -- (a0);
\end{tikzpicture}
&
\begin{tikzpicture}[x=.9cm]
  \path (-3,0)  node (w2p7) {$\omegav_2 + \omegav_7$};
  \path (-4,-2) node (w3) {$\omegav_3$};
  \path (-2,-4) node (w72) {$2\omegav_7$};
  \path (-3,-5) node (w6) {$\omegav_6$};
  \path (-3,-6) node (w1) {$\omegav_1$};
  \path (-3,-7) node (w0) {$0$};

  \path (0,0)   node (a2a1) {$A_2+A_1$};
  \path (1,-1)  node (a14)  {$4A_1$};
  \path (-1,-2)  node (a2)   {$A_2$};
  \path (0,-3)  node (a13p) {$(3A_1)'$};
  \path (1,-4)  node[reed1] (a13s) {$(3A_1)''$};
  \path (0,-5)  node[reed1] (a12)  {$2A_1$};
  \path (0,-6)  node[reed1] (a1)   {$A_1$};
  \path (0,-7)  node[reed1] (a0)   {$1$};
  
  \path (0.5,-0.5) coordinate[reed2] (e1);
  \path (-0.5,-2.5) coordinate[reed2] (e2);
  
  \draw (a2a1) -- (a2) -- (a13p) -- (a12) -- (a1) -- (a0)
        (a2a1) -- (a14) -- (a13s) -- (a12)
        (a14) -- (a13p)
        (a2) .. controls(-1,-4) .. (a12);
  \draw (w2p7) -- (w3) .. controls(-4,-4).. (w6) -- (w1) -- (w0)
        (w2p7) .. controls(-2,-1) .. (w72) -- (w6);
  \draw[pimap] (w2p7) .. controls(-2,0) .. (e1);
  \draw[pimap] (w3) .. controls(-3,-2) .. (e2);
  \draw[pimap] (w72) -- (a13s);
  \draw[pimap] (w6) -- (a12);
  \draw[pimap] (w1) -- (a1);
  \draw[pimap] (w0) -- (a0);
\end{tikzpicture}
\\
$E_6$ & $E_7$
\end{tabular}
\bigskip

\begin{tabular}{ccc}
\noindent
\begin{tikzpicture}[x=0.9cm]
  \path (-2,0)  node (w2) {$\omegav_2$};
  \path (-2,-2) node (w7) {$\omegav_7$};
  \path (-2,-4) node (w1) {$\omegav_1$};
  \path (-2,-5) node (w8) {$\omegav_8$};
  \path (-2,-6) node (w0) {$0$};
  
  \path (0,0)   node (a2a1) {$A_2+A_1$};
  \path (1,-1)  node (a14)  {$4A_1$};
  \path (0,-2)  node (a2)   {$A_2$};
  \path (1,-3)  node (a13)  {$3A_1$};
  \path (0,-4)  node[reed1] (a12)  {$2A_1$};
  \path (0,-5)  node[reed1] (a1)   {$A_1$};
  \path (0,-6)  node[reed1] (a0)   {$1$};
  
  \path (0.5,-0.5) coordinate[reed2] (e1);
  \path (0.5,-2.5) coordinate[reed2] (e2);
  
  \draw (a2a1) -- (a2) -- (a12) -- (a1) -- (a0)
        (a2a1) -- (a14) -- (a13) -- (a12)
        (a2) -- (a13);
  \draw (w2) -- (w7) -- (w1) -- (w8) -- (w0);
  \draw[pimap] (w2) .. controls(-1,0) .. (e1);
  \draw[pimap] (w7) .. controls(-1,-2) .. (e2);
  \draw[pimap] (w1) -- (a12);
  \draw[pimap] (w8) -- (a1);
  \draw[pimap] (w0) -- (a0);
\end{tikzpicture}
&
\begin{tikzpicture}[x=0.9cm]
  \path (-2,0)  node (w3) {$\omegav_2$};
  \path (-2,-2) node (w1) {$\omegav_4$};
  \path (-2,-3) node (w4) {$\omegav_1$};
  \path (-2,-4) node (w0) {$0$};

  \path (0,0)   node (a2)   {$A_2$};
  \path (1,-1)  node (a1a1) {$A_1+\widetilde{A}_1$};
  \path (0,-2)  node[reed1] (a1t)  {$\widetilde{A}_1$};
  \path (0,-3)  node[reed1] (a1)   {$A_1$};
  \path (0,-4)  node[reed1] (a0)   {$1$};

  \path (0.5,-0.5) coordinate[reed2] (e1);

  \draw (a2) -- (a1a1) -- (a1t) -- (a1) -- (a0)
        (a2) -- (a1t);
  \draw (w3) -- (w1) -- (w4) -- (w0);
  \draw[pimap] (w3) .. controls(-1,0) .. (e1);
  \draw[pimap] (w1) -- (a1t);
  \draw[pimap] (w4) -- (a1);
  \draw[pimap] (w0) -- (a0);
\end{tikzpicture}
&
\begin{tikzpicture}[x=0.9cm]
  \path (-2,0)  node (w2) {$\omegav_1$};
  \path (-2,-2) node (w1) {$\omegav_2$};
  \path (-2,-3) node (w0) {$0$};

  \path (0,0)   node (g2)   {$G_2(a_1)$};
  \path (1,-1)  node (a1t) {$\widetilde{A}_1$};
  \path (0,-2)  node[reed1] (a1)   {$A_1$};
  \path (0,-3)  node[reed1] (a0)   {$1$};

  \path (0.5,-0.5) coordinate[reed2] (e1);

  \draw (g2) -- (a1t) -- (a1) -- (a0)
        (g2) -- (a1);
  \draw (w2) -- (w1) -- (w0);
  \draw[pimap] (w2) .. controls(-1,0) .. (e1);
  \draw[pimap] (w1) -- (a1);
  \draw[pimap] (w0) -- (a0);
\end{tikzpicture}
\\
$E_8$ & $F_4$ & $G_2$
\end{tabular}
\bigskip
\end{center}
\caption{$\Grsm$, $\cNsm$, and Reeder pieces in the exceptional types}\label{tab:excpo}
\end{table}

\subsection{Types $F_4$ and $G_2$}
\label{subsect:fg}

The poset $\Lambdavsm^+$ for each of these types is displayed in Table~\ref{tab:excpo}. Note that
we are numbering the nodes of the Dynkin diagram for $G$ as in \cite[Plates VIII, IX]{bourbaki},
which means that the fundamental weights of $\Gv$ are numbered in the reverse of what would be the natural
order if we were considering $\Gv$ alone.
The involution $\lambdav\mapsto-w_0\lambdav$ is trivial in these types. By inspection, we have:
\begin{lem} \label{lem:FG-Mirred}
$\cM$ is irreducible. \qed
\end{lem}

Groups of these types arise by `folding': each such $G$ is the set of fixed points of an automorphism $\sigma$ of some larger simply-connected simple algebraic group $H$ of simply-laced type, where $\sigma$ comes from an automorphism of the Dynkin diagram of $H$.  The type of $H$ is given in the following table.
\[
\begin{array}{c|c}
G & H \\
\hline
F_4 & E_6 \\
G_2 & D_4
\end{array}
\]
The inclusion $G \hookrightarrow H$ induces embeddings $\Gr \hookrightarrow \Gr_H$ and $\cN \hookrightarrow \cN_H$.

In both cases, Proposition~\ref{prop:general} is already known for $H$.  As in Section~\ref{subsect:e}, we will deduce Proposition~\ref{prop:general} for $G$ from the result for $H$, but since $H$ is now bigger than $G$, the arguments will be much easier.  We begin once again with some computations, recorded in Table~\ref{tab:FGcalc}.   

\begin{enumerate}
\item {\it For each $\lambdav\in\Lambdavsm^+$, compute $\dim\Gr_\lambdav$.} As before, we use the formula $\dim\Gr_\lambdav=\langle\lambdav,2\rho\rangle$. 

\item {\it For each $\lambdav\in\Lambdavsm^+$, find the $\HO$-orbit $\Gr_{H,\muv}$ containing $\Gr_\lambdav$.}  To write down coweights of $H$, we choose a $\sigma$-stable maximal torus $T_H$ such that $T_H^\sigma = T$. With a suitable choice of positive system, the desired coweight $\muv$ is simply $\lambdav$ regarded as a $\sigma$-stable element of $\Lambdav_H^+$. Crucially, we observe that in each case $\muv\in\Lambdav_{H,\sm}^+$. 

\item {\it For this $\HO$-orbit $\Gr_{H,\muv}$, find the nilpotent orbits in $\cN_{H,\sm}$ contained in $\pi_H(\cM_{H,\muv})$.} We refer to Table~\ref{tab:excpo} for $H$ of type $E_6$, and to Table~\ref{tab:classcalc} or~\ref{tab:classpo} for $H$ of type $D_4$.

\item {\it For each nilpotent orbit in $\cN_{H,\sm}$, compute its intersection with $\cN$.}  Recall that by Dynkin--Kostant theory, given a nilpotent orbit $C \subset \cN$, one can associate to it a `weighted Dynkin diagram', which is really a coweight $\nuv(C): \C^\times \to T$ obtained by restricting a certain homomorphism $SL_2 \to G$ to the subgroup $\{[\begin{smallmatrix} a & \\ & a^{-1} \end{smallmatrix}]\} \cong \C^\times$.  The same remarks apply to $H$.  For nilpotent orbits $C \subset \cN$ and $C' \subset \cN_H$, we have $C \subset C'$ if and only if 
$\nuv(C)$, when regarded as a $\sigma$-stable element of $\Lambdav_H^+$, equals $\nuv(C')$.  For each $C'$, we can find $C$ and $\dim C$ by consulting the tables of weighted Dynkin diagrams 
in~\cite[Section 13.1]{carter}.
\end{enumerate}

\begin{table}
\[\small
\begin{array}{@{}cc@{}}
F_4: &
\begin{array}{cc|c|c|cc}
\lambdav & \dim \Gr_\lambdav &
  \muv :  \Gr_\lambdav \subset \Gr_{H,\muv} &
  \pi_H(\cM_{H,\muv}) &
  \cN \cap \pi_H(\cM_{H,\muv}) & \dim  \\
\hline
\omegav_2 & 30 &
   \omegav_{H,4} & A_2 & A_2 & 30  \\
 &&              & 3A_1 & A_1 + \widetilde{A}_1 & 28 \\           
\hline
\omegav_4 & 22 &
   \omegav_{H,1} + \omegav_{H,6} & 2A_1 & \widetilde{A}_1 & 22 \\
\hline
\omegav_1 & 16 &
   \omegav_{H,2} & A_1 & A_1 & 16 \\
\hline
0 & 0 & 0 & 1 & 1 & 0
\end{array}
\\
\ \\ \ \\
G_2: &
\begin{array}{cc|c|c|cc}
\lambdav & \dim \Gr_\lambdav &
  \muv : \Gr_\lambdav \subset \Gr_{H,\muv} &
  \pi_H(\cM_{H,\muv}) &
  \cN \cap \pi_H(\cM_{H,\muv}) & \dim  \\
\hline
\omegav_1 & 10 &
   (2,1,1,0) & [3^2 1^2] & G_2(a_1) & 10  \\
 &&          & [3 2^2 1] & \widetilde{A}_1 & 8 \\           
\hline
\omegav_2 & 6 &
   (1,1,0,0) & [2^2 1^4] & A_1 & 6 \\
\hline
0 & 0 & 0 & [1^8] & 1 & 0
\end{array} \\
\ 
\end{array}
\]
\caption{Orbit calculations for types $F_4$ and $G_2$.}\label{tab:FGcalc}
\end{table}

\begin{lem}\label{lem:FG-cnsm}
The image $\cNsm=\pi(\cM)$ is an irreducible closed subset of $\cN$, and the map $\pi: \cM \to \cNsm$ is finite.  For each $\Gr_\lambdav \subset \Grsm$, the image $\pi(\cM_\lambdav)$ consists precisely of the nilpotent orbits listed in Table~\ref{tab:FGcalc}.  Finally, the commutative diagram
\begin{equation}\label{eqn:FG-cart}
\vcenter{\xymatrix{
\cM \ar[r]\ar[d]_{\pi} & \cM_H \ar[d]^{\pi_H} \\
\cNsm \ar[r] & \cN_{H,\sm}
}}
\end{equation}
is cartesian.
\end{lem}
\begin{proof}
The calculations leading to Table~\ref{tab:FGcalc} show that under the embedding $\Grom \hookrightarrow \Gr_{H,0}^-$, we have $\cM \subset \cM_H$.  It follows from Lemma~\ref{lem:commute} that $\pi(\cM)$ is contained in $\fg \cap \cN_H = \cN$.  Moreover, $\pi$ is finite because $\pi_H$ is finite, and therefore $\cNsm=\pi(\cM)$ is closed in $\cN$.
Since $\cM$ is irreducible, $\cNsm$ is also irreducible.  

It also follows from Lemma~\ref{lem:commute} that for each $\lambdav \in \Lambdavsm^+$, the image $\pi(\cM_\lambdav)$ must be contained in the union of the nilpotent orbits listed for $\lambdav$ in Table~\ref{tab:FGcalc}.  The image $\pi(\cM_\lambdav)$ is $G$-stable and nonempty, so in cases where only one nilpotent orbit is listed, 
it is automatic that $\pi(\cM_\lambdav)$ equals that orbit.  It remains to consider the case where $\Gr_\lambdav$ is the largest $\GO$-orbit in $\cM$: this is listed with two nilpotent orbits in each type.  Let us denote these by $C_1$ and $C_2$, with $C_2 \subset \overline{C_1}$ (see \cite[Section 13.4]{carter} for the closure order).  Because $\pi$ is finite, we have $\dim(\pi(\cM_\lambdav)) = \dim \Gr_\lambdav$.  Since $\dim C_1 = \dim \Gr_\lambdav > \dim C_2$, we must have $C_1 \subset \pi(\cM_\lambdav)$.  Since $\pi(\cM)$ is closed, we must also have $C_2 \subset \pi(\cM)$, but we already know that $C_2 \not\subset \pi(\cM_\nuv)$ for any smaller orbit $\Gr_\nuv$.  Thus, $C_2 \subset \pi(\cM_\lambdav)$.

Finally, consider an element $x \in \cNsm$.  We obviously have $\pi^{-1}(x) \subset \pi_H^{-1}(x)$.  Moreover, we know that $\pi^{-1}(x)$ is nonempty, so it is a union of $\Z/2\Z$-orbits.  On the other hand, by Theorem~\ref{thm:mn} for $H$, $\pi_H^{-1}(x)$ contains exactly one $\Z/2\Z$-orbit, so we conclude that $\pi^{-1}(x) = \pi_H^{-1}(x)$.  This means that the diagram~\eqref{eqn:FG-cart} is cartesian.
\end{proof}

\begin{proof}[Proof of Proposition~\ref{prop:general} for types $F_4$ and $G_2$]
We noted in Lemmas~\ref{lem:FG-Mirred} and~\ref{lem:FG-cnsm} that $\cM$ and $\cNsm$ are irre\-duc\-ible.  By Lemma~\ref{lem:FG-cnsm} and inspection of Table~\ref{tab:FGcalc}, we see that the Reeder pieces each consist of one or two nilpotent orbits, that they form a partition of $\cNsm$, and that they are in bijection with the set of $\GO$-orbits in $\Grsm$.  Since every $\GO$-orbit in $\Gr$ is $\iota$-stable, we have established the bijection~\eqref{eqn:gen-bij}.

Suppose $S$ is a Reeder piece consisting of a single nilpotent orbit $C$.  Let $C_H$ denote the 
$H$-saturation of $C$, listed in Table~\ref{tab:FGcalc}.  Referring to Tables~\ref{tab:classcalc} 
and~\ref{tab:excpo} and invoking Proposition~\ref{prop:general} for $H$, we see that in each case, the map $\pi_H^{-1}(C_H) \to C_H$ is an isomorphism.  Because~\eqref{eqn:FG-cart} is cartesian, the map $\pi^{-1}(C) \to C$ is an isomorphism as well.

Next, consider a Reeder piece $S$ containing two nilpotent orbits $C_1$ and $C_2$, with $C_2 \subset \overline{C_1}$.  The reasoning of the preceding paragraph applies verbatim to $C_2$ and shows that $\pi^{-1}(C_2) \to C_2$ is an isomorphism.  On the other hand, for $x \in C_1$, the fibre $\pi_H^{-1}(x)$ always contains two points, so $\pi^{-1}(C_1) \to C_1$ must be a $2$-fold cover.  Since $\pi^{-1}(C_1)$ is a dense subset of a single $\iota$-stable $\GO$-orbit, it must be connected.

Finally, the assertions about the Springer correspondence follow from Reeder's calculations 
of $V_\lambdav^{\Tv}\otimes\epsilon$ in~\cite[Table 5.1]{reeder3}, which have been reproduced 
in Table~\ref{tab:exczw}, and the description of the Springer correspondence in 
\cite[Section 13.3]{carter}. (When comparing Table~\ref{tab:exczw} with Reeder's table, 
note that the identification of the Weyl groups of $G$ and $\Gv$ interchanges
${}'$ and ${}''$ in the labels of irreducible representations of $W$, and that Reeder numbers the nodes
of the Dynkin diagram differently.)
\end{proof}

\begin{rmk}\label{rmk:g2-spr}
Suppose $G$ is of type $G_2$, and let $C$ denote the nilpotent orbit labelled $G_2(a_1)$.  This orbit has one nontrivial rank-$1$ local system $\sigma$.  The pair $(C,\sigma)$ does not occur in the Springer correspondence, so 
for that orbit, the term involving $\IC(\overline{C},\sigma)$ should be omitted from the right-hand side of the formula~\eqref{eqn:zspring-double}.
\end{rmk} 

\section{Consequences}
\label{sect:conseq}

\subsection{Normality and seminormality}
\label{subsect:normal}

The question of which orbit closures in $\cN$ are normal has been completely answered in all types other than $E_7$ and $E_8$ ~\cite{kostant,kp:gc,bs,kraft,broer:f4,broer:dv,sommers}. As a special case of such an orbit closure, $\cNsm$ is known to be normal in all types except the following.
\begin{itemize}
\item In type $D_n$ for odd $n\geq 5$, $\cNsm$ (the closure of the orbit $[3^22^{n-3}]$) is not normal, but is known to be seminormal~\cite{kp:gc}. Recall that a variety is \emph{seminormal} if every bijective map to it is an isomorphism, and that normality implies seminormality. 
\item In type $E_6$, $\cNsm$ (the closure of the orbit $2A_2$) is not normal, and indeed its normalization map is not bijective~\cite[Section 5, (F)]{bs}.
\item In types $E_7$ and $E_8$, $\cNsm$ (the closure of the orbit $A_2+A_1$) is expected to be normal, though this has not been proved~\cite[Remark 7.9]{broer:dv}. The normalization map of $\cNsm$ is known to be bijective: by~\cite[Section 5, (E)]{bs}, this follows from the fact that $\dim\cH_x^{-\dim\cNsm}(\IC(\cNsm))=1$ for all $x\in\cNsm$ (see Table~\ref{tab:Estalk}).
\end{itemize}
  
\begin{conj}\label{conj:E}
If $G$ is of type $E$, the variety $\cNsm$ is seminormal.
\end{conj}

Note that in types $E_7$ and $E_8$, Conjecture~\ref{conj:E} is equivalent to the conjecture that $\cNsm$ is normal. The following is an immediate consequence of Theorem~\ref{thm:mn}. 

\begin{prop}\label{prop:isom}
Assume either that $G$ is not of type $E$, or that Conjecture~\ref{conj:E} holds.  Then the map $\cM/(\Z/2\Z) \to \cNsm$ induced by $\pi$ is an isomorphism of varieties. \qed
\end{prop}

We can also use Theorem~\ref{thm:mn} to construct the normalization of $\cNsm$, or, more generally, of the closure of any Reeder piece.

\begin{prop}\label{prop:normalization}
Let $\lambdav\in\Lambdavsm^+$, and let $C$ be the open orbit in the Reeder piece $\pi(\cM_\lambdav)$.
\begin{enumerate}
\item If $\lambdav\neq-w_0\lambdav$, then the map $\overline{\Gr_\lambdav} \cap \Grom\to\overline{C}$ induced by $\pi$ is the normalization map of $\overline{C}$.
\item If $\lambdav=-w_0\lambdav$, then the bijection $(\overline{\Gr_\lambdav} \cap \Grom)/(\Z/2\Z)\to\overline{C}$ induced by $\pi$ is the normalization map of $\overline{C}$.
\end{enumerate}
\end{prop}
\begin{proof}
The variety $\overline{\Gr_\lambdav} \cap \Grom$ is normal, because it is an open subset of the affine Schubert variety $\overline{\Gr_\lambdav}$. In case (1), we know from Proposition~\ref{prop:general} that $\pi$ induces an isomorphism from $\cM_\lambdav$ to $C$; since $\pi$ is finite, the claim follows. In case (2), we know from Lemma~\ref{lem:iota-dual} that $\overline{\Gr_\lambdav} \cap \Grom$ is $\iota$-stable. Since quotients by finite group actions preserve normality, the claim follows.
\end{proof}

We can deduce a sort of converse to Proposition~\ref{prop:isom} in types $E_7$ and $E_8$.

\begin{cor}\label{cor:normal}
Suppose that $G$ is of type $E_7$ or $E_8$. If the map $\cM/(\Z/2\Z) \to \cNsm$ induced by $\pi$ is an isomorphism, then the closure of every Reeder piece is normal.
\end{cor}

Note that, apart from $\cNsm$ itself, the only closure of a Reeder piece in these types which is not known to be normal is the closure of $A_2$.

\subsection{Special pieces}
\label{subsect:special}

Recall that a \emph{special piece} is a subset of $\cN$ obtained by taking the closure of a special nilpotent orbit and deleting from it the closures of all strictly smaller special nilpotent orbits.  The special pieces are locally closed and form a partition of $\cN$. See \cite[Section 13.4]{carter} for details of the special pieces in each type.

Lusztig has conjectured~\cite[Section 0.6]{lusztig:notes} that for each special piece $S$, there is a smooth variety $\tilde S$ with commuting actions of $G$ and a specified finite group $A_S$, as well as a $G$-equivariant isomorphism $\tilde S/A_S \cong S$.  The conjecture also proposes a specific relationship between $G$-orbits in $S$ and stabilizers in $A_S$: in the case where $S$ consists of two orbits $C_1$ and $C_2$ with $C_1$ special, the group $A_S$ is $\Z/2\Z$, and the claim is that $C_2$ is the image of the $(\Z/2\Z)$-fixed subvariety of $\tilde S$.

Lusztig formulated his conjecture only for the exceptional types, because in the classical types the result was known by the work of Kraft--Procesi~\cite{kp:sd}. In unpublished work, Lusztig has verified the conjecture in type $G_2$. If a variety $\tilde S$ satisfying the conditions in Lusztig's conjecture exists, it is known to be unique~\cite{as}.

Now in the simply-laced types, every Reeder piece is a special piece (by inspection of Tables~\ref{tab:classcalc} and~\ref{tab:excpo}). For a Reeder piece $\pi(\cM_\lambdav)$ consisting of two orbits, Proposition~\ref{prop:general} supplies a smooth variety $\cM_\lambdav$ with commuting actions of $G$ and $\Z/2\Z$, and a $G$-equivariant bijection $\cM_\lambdav/(\Z/2\Z)\to \pi(\cM_\lambdav)$ in which the smaller orbit is the image of the $(\Z/2\Z)$-fixed subvariety $\cM_\lambdav^\iota$. If this bijection is an isomorphism, then Lusztig's conjecture is verified.

\begin{prop}\label{prop:lusztig}
Lusztig's conjecture holds for the smallest nontrivial special piece in type $E_6$, namely, that with special orbit $A_2$. If $G$ is of type $E_7$ or $E_8$ and Conjecture~\ref{conj:E} holds, then Lusztig's conjecture holds for the special pieces with special orbits $A_2+A_1$ and $A_2$.
\end{prop}
\begin{proof}
In type $E_6$, the closure of the orbit $A_2$ is normal~\cite{sommers}, so the special piece $S$ containing $A_2$ is also normal. Hence the bijection $\cM_{\omegav_4}/(\Z/2\Z)\to S$ is an isomorphism, as required. If $G$ is of type $E_7$ or $E_8$ and Conjecture~\ref{conj:E} holds, then the analogous bijection is an isomorphism by Proposition~\ref{prop:isom}.
\end{proof}

When $G$ is of non-simply-laced type, the Reeder pieces are the intersections with $\cNsm$ of Reeder pieces for the simply-laced group $H$ from which $G$ is obtained by `folding': see Lemma~\ref{lem:FG-cnsm} (the analogous result in types $B$ and $C$ also holds). Thus, the Reeder pieces for $G$ are related to special pieces for $H$ rather than to special pieces for $G$. For example, in type $F_4$ the nontrivial Reeder piece consists of two special orbits, whereas the two orbits in the smallest nontrivial special piece, $\widetilde A_1$ and $A_1$, are in separate Reeder pieces. 

\subsection{Relation to Reeder's results}
\label{subsect:reeder}

By now, it will be clear that many results of the present paper were inspired by 
Reeder's work~\cite{reeder1, reeder2, reeder3}, and in particular his
explicit computations of $V^{\Tv}$ for $V$ small. Here, we briefly explain how some of Reeder's other results fit into the context of the present paper.

\subsubsection{Big orbits and subdual orbits}

Inside the nilpotent cone $\cNv$ for $\Gv$, there is a unique maximal open subset consisting of special nilpotent $\Gv$-orbits.  Orbits contained in this set are said to be \emph{big}.  In the simply-laced types, Reeder proved that for a small representation $V$ of $\Gv$, the Weyl group action on $V^\Tv$ can be described in terms of Springer representations attached to some big nilpotent orbit~\cite{reeder1}.

This result is related to Theorem~\ref{thm:sss} by \emph{Spaltenstein duality}, an order-reversing map from nilpotent orbits in $\cNv$ to special nilpotent orbits in $\cN$.  The following facts can be verified by case-by-case calculations:
\begin{enumerate}
\item The Spaltenstein dual of a big orbit is the open orbit in some Reeder piece.  Moreover, every open orbit of a Reeder piece arises in this way.
\item Let $\check C$ be a big orbit in $\cNv$, and let $C \subset \cN$ be its Spaltenstein dual.  Then
\[
\bigoplus_{\substack{\text{$\check E$ a local}\\ \text{system on $\check C$}}} \Spr(\IC(\check C,\check E)) \cong \epsilon \otimes
\bigoplus_{\substack{\text{$E$ a local}\\ \text{system on $C$}}} \Spr(\IC(C,E)).
\]
\end{enumerate}
Using these facts, Reeder's main result in~\cite{reeder1} can be deduced from Theorem~\ref{thm:sss}.

The definition of `big orbit' still makes sense in the non-simply-laced types, but it is no longer such a well-behaved notion.  There are big orbits (for example, $F_4(a_3)$ in type $F_4$) whose Spaltenstein dual is not even contained in $\cNsm$; conversely, there are open orbits of Reeder pieces (e.g., the orbit $A_1$ in type $G_2$) that are not special, and so cannot be the Spaltenstein dual of anything.  It would be interesting to see whether the generalizations of Spaltenstein duality in~\cite{a:ord} or~\cite{sommers:lcq}, combined with some modification of the definition of `big', would allow the result of~\cite{reeder1} to be generalized to the non-simply-laced case.

\subsubsection{Small weighted Dynkin diagrams}

A nilpotent orbit $C\subset\cN$ is determined by its `weighted Dynkin diagram', which we can think of
as the unique coweight $\nuv(C)\in\Lambdav^+$ such that
$\nuv(C):\C^\times\to T$ extends to a homomorphism $\varphi_{\nuv(C)}: SL_2 \to G$ with 
$d\varphi_{\nuv(C)}([\begin{smallmatrix} 0 & 1 \\ 0 & 0 \end{smallmatrix}]) \in C$. For example, 
if $C_\mn$ is the minimal (nonzero) nilpotent orbit, then $\nuv(C_\mn)$ is the highest short
coroot $\alphav_0$. It is easy to check using \cite[Section 13.1]{carter} that $\lambdav\in\Lambdavsm^+$ arises as $\nuv(C)$ if and only if $\lambdav=-w_0\lambdav$. Reeder made this observation for the simply-laced types in~\cite[Proposition 3.2]{reeder3}, of which we can give the following generalization.

\begin{prop}\label{prop:attach}
Let $\lambdav\in\Lambdavsm^+$ be such that $\lambdav=-w_0\lambdav$, and let $C\subset\cN$ be such that
$\lambdav=\nuv(C)$. Then $C=\pi(\cM_{\lambdav}^\iota)$, and $\pi$ restricts to an isomorphism
$\cM_\lambdav^\iota\simto C$. In particular, $\pi$ restricts to an isomorphism $\cM_{\alphav_0}\simto C_\mn$.
\end{prop}
\begin{proof}
Applying Lemma~\ref{lem:commute} to the homomorphism $\varphi_{\nuv(C)}:SL_2\to G$, we see that
$\pi(\cM_\lambdav^\iota)$ meets $C$. But by Proposition~\ref{prop:general}, $\cM_\lambdav^\iota$ is a single
$G$-orbit and $\pi$ restricted to $\cM_\lambdav^\iota$ is injective, so $\pi$ must map $\cM_\lambdav^\iota$
isomorphically onto $C$. In case $\lambdav=\alphav_0$, we have $\cM_{\alphav_0}=\cM_{\alphav_0}^\iota$ because the
corresponding Reeder piece consists solely of $C_\mn$, as shown in Tables~\ref{tab:classcalc} and \ref{tab:excpo}. 
\end{proof}

\subsubsection{Subdual orbits}

Returning to the simply-laced case, Reeder introduces a notion of \emph{subdual orbit} in~\cite{reeder2}, defining it in terms of big orbits and Spaltenstein duality.  His definition is easily seen to be equivalent to the following: the subdual orbit corresponding to a small representation $V_\lambdav$ is the unique smallest nilpotent orbit in the Reeder piece $\pi(\cM_\lambdav)$.  Thus, if $\lambdav=-w_0\lambdav$ (that is, $V_\lambdav$ is self-dual), the subdual orbit is $\pi(\cM_\lambdav^\iota)$.

Reeder's results about subdual orbits are quite different in nature from the other results mentioned above.  They involve regular functions on nilpotent orbits, and so, implicitly, coherent sheaves, rather than perverse sheaves.  The authors do not know how to understand these results in the context of the present paper. Investigating the behaviour of coherent sheaves under the functor $\pi_*$ may be a fruitful avenue for future inquiry.

\subsection{Broer's covariant theorem}
\label{subsect:broer}

Let $\fgv$ denote the Lie algebra of $\Gv$, let $\ftv \subset \fgv$ denote the Lie algebra of $\Tv$, and let $\Coinv(W)$ denote the coinvariant ring of $W$.  Chevalley's restriction theorem states that the inclusion $\ftv \to \fgv$ induces an isomorphism of graded rings $\C[\fgv]^\Gv \to \C[\ftv]^W$.  Here, and below, $\C[X]$ denotes the ring of regular functions on $X$.  For convenience, we regard the gradings on polynomial rings like $\C[\fgv]$ and $\C[\ftv]$, as well as on $\Coinv(W)$, as being concentrated in even degrees.  All these rings are generated by their homogeneous elements of degree~$2$.

The following remarkable theorem of Broer generalizes Chevalley's restriction theorem.  Below, we explain how to deduce Broer's theorem from Theorem~\ref{thm:sss} (or rather, the case-free version given in Remark~\ref{rmk:sss-g2}). The problem of finding such a geometric approach to Broer's theorem was raised by Ginzburg, and was the main motivation for the present paper.

\begin{thm}[Broer]\cite{broer}\label{thm:broer}
Let $V$ be a small representation of $\Gv$.  Then there are natural graded isomorphisms
\begin{align}
(\C[\cNv] \otimes V)^\Gv &\cong (\Coinv(W) \otimes V^\Tv)^W, \label{eqn:broer-neq} \\
(\C[\fgv] \otimes V)^\Gv &\cong (\C[\ftv] \otimes V^\Tv)^W. \label{eqn:broer-eq}
\end{align}
\end{thm}
\begin{proof}
Let $i_\bo: \{\bo\} \hookrightarrow \cM$ and $i_0: \{0\} \hookrightarrow \cN$ denote the inclusion maps.  According to~\cite[Theorem~24.8(c)]{lusztig:cs5}, if $\cF \in \Perv_G(\cN)$ is a simple perverse sheaf not occurring in the Springer correspondence, then $i_0^*\cF = 0$.  The dual statement that $i_0^!\cF = 0$ also holds, because the Verdier dual of $\cF$ cannot occur in the Springer correspondence either.  Combining this observation with Remark~\ref{rmk:sss-g2}, we obtain an isomorphism
\begin{equation}\label{eqn:costalk-broer}
i_\bo^! \Sat(V) \cong
i_0^! \Spr(V^\Tv \otimes \epsilon)
\end{equation}
in the derived category of sheaves on a point.  Since $\pi$, $i_\bo$, and $i_0$ are all $G$-equivariant, the isomorphism~\eqref{eqn:costalk-broer} also holds in the $G$-equivariant derived category of a point.  The isomorphisms~\eqref{eqn:broer-neq} and~\eqref{eqn:broer-eq} will be obtained by computing the cohomology of both sides of~\eqref{eqn:costalk-broer} in the ordinary and equivariant derived categories, respectively.

For the left-hand side, note that the skyscraper sheaf $i_{\bo*}\uC$ is isomorphic to $\Sat(V_0)$, where 
$V_0$ is the trivial representation.  Ext-groups between perverse sheaves on $\Gr$ are described in~\cite[Proposition~1.10.4]{ginzburg} in terms of $\Gv$-equivariant modules over $\C[\cNv]$.  Using that result, we have
\begin{multline*}
H^n(i_\bo^!\Sat(V)) 
\cong \Hom(\uC, i_\bo^!\Sat(V)[n])\\ \cong \Hom(\Sat(V_0), \Sat(V)[n]) 
\cong \Hom_{\Gv,\C[\cNv]}(\C[\cNv], \C[\cNv] \otimes V\la n\ra),
\end{multline*}
where $\la n\ra$ denotes a shift of grading by $n$ on a graded module.  Thus,
\begin{equation}
H^\bullet(i_\bo^!\Sat(V)) \cong (\C[\cNv] \otimes V)^\Gv.
\end{equation}
Similarly, for the right-hand side, we have $i_{0*}\uC \cong \Spr(\epsilon)$, so
\begin{multline*}
H^n(i_0^!\Spr(V^\Tv \otimes \epsilon))
 \cong \Hom(\uC, i_0^!\Spr(V^\Tv \otimes \epsilon)[n]) \\ \cong \Hom(\Spr(\epsilon), \Spr(V^\Tv \otimes \epsilon)[n])
\cong (\epsilon \otimes \Coinv(W) \otimes (V^\Tv \otimes \epsilon)\la n\ra)^W,
\end{multline*}
where the last step uses the computation of Ext-groups on $\cN$ in~\cite[Theorem~4.6]{a:gfhl}.  We conclude that
\begin{equation}
H^\bullet(i_0^!\Spr(V^\Tv \otimes \epsilon)) \cong (\Coinv(W) \otimes V^\Tv)^W,
\end{equation}
and then~\eqref{eqn:broer-neq} follows.

Finally, it is known that the cohomology of both sides of~\eqref{eqn:costalk-broer} vanishes in odd degrees.  From this, it can be deduced that both sides of~\eqref{eqn:costalk-broer} are semisimple objects in the $G$-equivariant derived category of a point.  For such an object $\cF$, the $G$-equivariant cohomology is simply given by $H^\bullet_G(\cF) \cong H^\bullet(\cF) \otimes H^\bullet_G(\uC)$.  Using the fact that $H^\bullet_G(\uC) \cong \C[\ftv]^W \cong \C[\fgv]^\Gv$, we have
\begin{align*}
H^\bullet_G&(i_\bo^!\Sat(V)) &
H^\bullet_G&(i_0^!\Spr(V^\Tv \otimes \epsilon)) \\
  &\cong (\C[\cNv] \otimes V)^\Gv \otimes \C[\fgv]^\Gv &
  &\cong (\Coinv(W) \otimes V^\Tv)^W \otimes \C[\ftv]^W \\
  &\cong (\C[\cNv] \otimes \C[\fgv]^\Gv \otimes V)^\Gv  &
  &\cong (\Coinv(W) \otimes \C[\ftv]^W \otimes V^\Tv)^W \\
  &\cong (\C[\fgv] \otimes V)^\Gv,  &
  &\cong (\C[\ftv] \otimes V^\Tv)^W.   
\end{align*}
In the last step, we have used the facts that $\C[\cNv] \otimes \C[\fgv]^\Gv \cong \C[\fgv]$ and that $\Coinv(W) \otimes \C[\ftv]^W \cong \C[\ftv]$.  (For $H^\bullet_G(i_\bo^!\Sat(V))$, we could have used~\cite{gr} or \cite[Theorem~4]{bf} instead.)  The isomorphism~\eqref{eqn:broer-eq} follows.
\end{proof}

\begin{rmk} \label{rmk:disappointment}
The proof of Remark~\ref{rmk:sss-g2} given in this paper relies on Reeder's computations of the $W$-action on $V^\Tv$, and in type $E$, Reeder used Broer's result to carry out the computation. So the above proof of Broer's result is circular in type $E$. The proof of Remark~\ref{rmk:sss-g2} given in \cite{ahr} does not rely on Reeder's computations, and thus avoids this circularity.
\end{rmk}

\end{document}